\documentclass{article}
\usepackage{amssymb}
\usepackage{amsthm}
\usepackage{amsmath}
\usepackage{tikz}
\usepackage{graphicx}
\usepackage{xcolor}
\usepackage{csquotes}
\usepackage{graphicx}
\usepackage{csquotes}
\usepackage{float}
\usepackage{soul}
\usepackage[a4paper,
left=2.5cm, right=2cm,
top=3cm, bottom=3cm]{geometry}
\definecolor{blue}{rgb}{0.0, 0.313, 0.608}
\usepackage{hyperref}
\hypersetup{
    colorlinks,
    linkcolor={blue},
    citecolor={blue},
    urlcolor={orange}
}

\newtheorem{theorem}{Theorem}[section]
\newtheorem{lemma}[theorem]{Lemma}
\newtheorem{remark}[theorem]{Remark}
\newtheorem{proposition}[theorem]{Proposition}
\newtheorem{corollary}[theorem]{Corollary}
\newtheorem{definition}[theorem]{Definition}
\numberwithin{equation}{section}

\newcommand\norm[1]{\left\lVert#1\right\rVert}
\newcommand{\R}{\mathbb{R}}
\newcommand{\N}{\mathbb{N}}	
\newcommand{\e}{\varepsilon}

\DeclareMathOperator{\ii}{\mathrm{i}}

\title{Lump solutions of the fractional Kadomtsev--Petviashvili equation}

\author{Handan Borluk\footnote{\texttt{handan.borluk@ozyegin.edu.tr}} \\  
{\small Ozyegin University, Department of Natural and Mathematical Sciences},
{\small  Istanbul,  Turkey} \\
\\
 Gabriele Bruell\footnote{\texttt{ gabriele.brull@math.lth.se}} \\ 
   {\small Lund University, Centre for Mathematical Sciences, Lund, Sweden}\\
\\
 Dag Nilsson\footnote{\texttt{dag.nilsson@math.lu.se}} \\ 
   {\small Lund University, Centre for Mathematical Sciences, Lund, Sweden}\\
}
\begin{document}

\maketitle

\begin{abstract}

Of concern are lump solutions for the fractional Kadomtsev--Petviashvili (fKP) equation. As in the classical Kadomtsev--Petviashvili equation, the fKP equation comes in two versions: fKP-I (strong surface tension case) and fKP-II (weak surface tension case). We prove the existence of nontrivial lump solutions for the fKP-I equation in the energy subcritical case $\alpha>\frac{4}{5}$ by means of variational methods. It is already known that there exist neither nontrivial lump solutions belonging to the energy space for the fKP-II equation \cite{deBouard0} nor for the fKP-I when $\alpha \leq  \frac{4}{5}$ \cite{linares1}. Furthermore, we show that for any $\alpha>\frac{4}{5}$ lump solutions for the fKP-I equation are smooth and decay quadratically at infinity. Numerical experiments are performed for the existence of lump solutions and their decay. Moreover, numerically, we observe cross-sectional symmetry of lump solutions for the fKP-I equation.

\end{abstract}


\renewcommand{\theequation}{\arabic{section}.\arabic{equation}}
\setcounter{equation}{0}
\section{Introduction}
The present paper is devoted to the study of fully localized solitary solutions (also known as \emph{lump solutions}) of the fractional Kadomtsev--Petviashvili (fKP) equation
\begin{equation}\label{eq:fKP}
	u_t+uu_x-\mathrm{D}^\alpha_x u_x+\sigma \partial_x^{-1}u_{yy}=0.
\end{equation}
Here the real function $u=u(t,x,y)$ depends on the spatial variable $(x,y)\in \R^2$ and the temporal variable $t\in \R_+$.
The linear operator $\mathrm{D}^\alpha_x$ denotes the Riesz potential of order $\alpha\in \R$ in $x$-direction, which is defined by multiplication with $|\cdot|^\alpha$ on the frequency space, that is
\begin{equation*}
	\mathcal{F}(\mathrm{D}^\alpha_x f)(t,\xi_1, \xi_2)=|\xi_1|^\alpha \hat{f}(t,\xi_1, \xi_2),
\end{equation*}
where the operator $\mathcal{F}$ denotes the extension to the space of tempered distributions $\mathcal{S}'({\R^n})$
of the Fourier transform
\[
\mathcal{F}(f)(\xi):=\int_{\R^n} f(x)e^{-\mathrm{i}\xi x}\, dx
\]
on the Schwartz space $\mathcal{S}({\R^n})$ with inverse  $\mathcal{F}^{-1}(f):=\frac{1}{2\pi}\mathcal{F}(f)(-\cdot)$. We also write $\hat f:=\mathcal{F}(f)$.
The operator $\partial_x^{-1}$ is defined as a Fourier multiplier operator on the $x$-variable as
$
\mathcal{F}(\partial_x^{-1}f)(t,\xi_1,\xi_2)=\frac{1}{\mathrm{i}\xi_1}\hat f(t,\xi_1,\xi_2).
$
In the case $\alpha=2$  equation \eqref{eq:fKP} becomes the classical Kadomtsev--Petviashvili (KP) equation
which was introduced by Kadomtsev \& Petviashvili \cite{kp} as a weakly two-dimensional extension of the celebrated Korteweg--de Vries (KdV) equation,
\begin{equation*}
	u_t+uu_x+u_{xxx}=0,
\end{equation*}
which is a spatially one-dimensional equation appearing in the context of small-amplitude shallow water-wave model equations. The KP equation comes in two versions: For $\sigma=-1$ it is called KP-I and for $\sigma=1$ it is called KP-II. Roughly speaking, the KP-I equation represents the case of strong surface tension, while the KP-II equation appears as a model equation for weak surface tension. 
Analogously to the classical case, the  fKP  equation is a two-dimensional extension of the fractional Korteweg--de Vries (fKdV) equation
\begin{equation*}
	u_t+uu_x-\mathrm{D}^\alpha_x u_x=0
\end{equation*}
and  \eqref{eq:fKP} is referred to as the fKP-I equation when  $\sigma=-1$ and as the fKP-II equation when $\sigma=1$.
Notice that for
$\alpha =1$ in \eqref{eq:fKP} we recover the KP-version of the Benjamin--Ono equation. During the last decade there has been a growing  interest in fractional regimes such as the fKdV or the fKP equation (see for example \cite{albert1, BBN, Eychenne, fonseca, franklenzmann, Kenig, klein, linares1, linares2, MS, natali, pava} and the references therein). Even though most of these equations are not derived by asymptotic expansions from governing equations  in fluid dynamics, they can be thought of as dispersive corrections.

\medskip

Formally, the fKP equation does not only conserve the $L^2$--norm
\[
M(u)=\int_{\R^2} u^2\,\mathrm{d}(x,y),
\]
but also the energy
\[
E_\alpha(u):=\int_{\R^2}\left( \frac{1}{2} (\mathrm{D}^\frac{\alpha}{2}_x u)^2-\tfrac{1}{6}u^3-\tfrac{1}{2}\sigma(\partial_x^{-1}u_y)^2 \right)\, \mathrm{d}(x,y).
\]
Notice that the corresponding energy space
\[
X_\frac{\alpha}{2}(\R^2):=\{u\in L^2(\R^2)\mid \mathrm{D}^\frac{\alpha}{2}_xu , \partial_x^{-1}u_y\in L^2(\R^2)\}
\]
equipped with the norm 
\begin{equation*}
||\phi||_\frac{\alpha}{2}^2:=\norm{\phi}_{L^2(\R^2)}^2+\norm{\mathrm{D}_x^\frac{\alpha}{2}\phi}_{L^2(\R^2)}^2+\norm{\partial_x^{-1}\partial_y\phi}_{L^2(\R^2)}^2,
\end{equation*}
includes a zero-mass constraint with respect to $x$.
We refer to \cite{linares1} for derivation issues and well-posedness results for the Cauchy problem associated with \eqref{eq:fKP}. The fKP equation is invariant under the scaling
\[
u_\lambda(t,x,y)=\lambda^\alpha u(\lambda^{\alpha+1}t,\lambda x, \lambda^{\frac{\alpha+2}{2}}y),
\]
and $\|u_\lambda\|_{L^2}=\lambda^{\frac{3\alpha-4}{4}}\|u\|_{L^2}$. Thus, $\alpha=\frac{4}{3}$ is the \emph{$L^2$-critical exponent} for the fKP equation. The ranges $\alpha>\frac{4}{3}$  and $\alpha<\frac{4}{3}$ are called \emph{sub}- and \emph{supercritical}, respectively. Due to the embedding $X_{\frac{\alpha}{2}}\subset L^3(\R^2)$ for $\alpha\geq \frac{4}{5}$ (cf. \cite[Lemma 1.1]{linares1}), we call $\alpha=\frac{4}{5}$ the \emph{energy critical exponent} for the fKP equation.

\medskip

 A traveling-wave solution $u(t,x,y)=\phi(x-ct,y)$ of the fKP equation propagating in $x$-direction with wave speed $c>0$, satisfies  the steady equation
\begin{equation}\label{eq:steady_KPI}
-c\phi +\frac{1}{2}\phi^2-\mathrm{D}^\alpha_x \phi+\sigma \partial_x^{-2}\phi_{yy}=0.
\end{equation}

A \textit{lump solution} is a traveling-wave solution which decays to $0$ as $|(x,y)|\rightarrow\infty$.



\medskip

\subsection{Main results}
Our aim is to study the existence and spatial decay of lump solutions for the fKP equation. Since it is known \cite{deBouard0, linares1}
that the fKP-II equation for any $\alpha$ as well as the fKP-I equation for $\alpha\leq \frac{4}{5}$ do not admit any lump solutions in $X_\frac{\alpha}{2}\cap L^3(\R^2)$,
the study of this paper is concerned with traveling waves for the fKP-I equation 
for $\alpha>\frac{4}{5}$.
We prove the following two main theorems. Moreover, we study lump solutions and some of their properties numerically. 

\begin{theorem}[Existence of lump solutions]\label{existence_theorem}
For any $\tfrac{4}{5}< \alpha$ there exists a lump solution $\phi\in X_\frac{\alpha}{2}$ of \eqref{eq:steady_KPI} with $\sigma=-1$.
\end{theorem}

\begin{theorem}[Decay of lump solutions]\label{thm:decay}  
Any lump solution $\phi \in X_{\frac{\alpha}{2}}$ of \eqref{eq:steady_KPI} with $\sigma=-1$ is smooth and satisfies
\[
r^2\phi \in L^\infty(\R^2),\qquad \mbox{where}\quad r^2(x,y)=x^2+y^2.
\]
\end{theorem}

\medskip

The classical KP-I equation possesses an explicit lump solution of the form
\begin{equation}\label{exact-solution}
\phi_e(x-ct,y)= 8c\frac{1-\tfrac{c}{3}(x-ct)^2+\tfrac{c^2}{3}y^2}{\left(1+\tfrac{c}{3}(x-ct)^2+\tfrac{c^2}{3}y^2\right)^2}.
\end{equation}

We would like to point out that de Bouard \& Saut studied the existence of of lump solutions for the generalized KP-I equation
\begin{equation}\label{KP-classical}
(u_t+u^pu_x-u_{xxx})_x-u_{yy}=0
\end{equation}
where $p=m/n\geq 1$, $m,n$ relatively prime and $n$ odd in \cite{deBouard0}. Furthermore, in their continuation paper \cite{deBouard}, de Bouard \& Saut investigated the symmetry and decay of lump solutions for \eqref{KP-classical} and showed that for all $p\geq 1$ the decay is quadratic. Our studies follow a similar approach as in \cite{deBouard0, deBouard}. However, special attention needs to be given to the nonlocal operator $\mathrm{D}_x^\alpha$. While many proofs can be adapted with a bit more technical effort due to the nonlocal operator, the result on decay of lump solutions in the supercritical case $\frac{4}{5}<\alpha<\frac{4}{3}$ (which includes the Benjamin--Ono KP version for $\alpha=1$) needs a modified approach, since in the supercritical case the symbol of an operator related to the linear dispersion is no longer $L^2$-integrable.

\medskip

\textbf{On the existence result:} We give a brief outline of the existence proof for lump solutions of \eqref{KP-classical} in \cite{deBouard0} by variational methods, since we will be using the same strategy to prove existence of lump solutions of the fKP-I equation \eqref{eq:fKP}. First consider the constrained minimization problem 
\begin{equation*}
I_\mu=\inf\left\{\norm{\phi}_Y^2\colon \phi\in Y , \int_{\R^2}\phi^{p+2}=\mu\right\}
\end{equation*}
for $\mu>0$ fixed, where $Y$ is the closure of $\partial_x(C_0^\infty(\mathbb{R}^2)$ with respect to the norm 
\begin{equation*}
\norm{\partial_x\varphi}_Y^2=\norm{\nabla\varphi}_{L^2(\R^2)}^2+\norm{\partial_x^2\varphi}_{L^2(\R^2)}^2.    
\end{equation*}
Via the Lagrange multiplier principle one finds (after rescaling) that solutions of the constrained minimization problem $I_\mu$ are lump solutions of \eqref{KP-classical}. The task is then to prove existence of solutions of $I_\mu$ and this is achieved using the concentration-compactness theorem (cf. Theorem \ref{cc}). The variational formulation associated with $I_\mu$ has several good properties. The functional being minimized is just the norm of the space $Y$. It is therefore immediate that it is coercive, bounded from below and weakly lower semi-continuous; properties which are all advantageous in the context of minimization problems, see \cite[Theorem 1.2]{struwe}. Furthermore, since the norm is homogeneous, it is easily shown that $I_\mu$ is subadditive as a function of $\mu$ and this property is essential in proving that the dichotomy scenario in the concentration-compactness theorem does not occur.  

\medskip

We prove Theorem \ref{existence_theorem} by extending the strategy of \cite{deBouard0},
outlined above, 
to the fractional case. Generally speaking, the fractional derivative and the fact that we are allowing for weak dispersion makes the proof of Theorem \ref{existence_theorem} more technical than its classical local counterpart $(\alpha=2)$. A key ingredient in the proof is the anisotropic Sobolev inequality \cite[Lemma 1.1]{linares1} (see also Proposition \ref{technical_results} (ii)), which in particular says that for $\tfrac{4}{5}\leq \alpha$, the space $X_\frac{\alpha}{2}$ is continuously embedded in $L^3(\mathbb{R}^2)$. This result is what determines the values of $\alpha$ for which we can prove existence of solitary waves. In fact, for $\alpha\leq \tfrac{4}{5}$ there exist no nontrivial lump solutions of for the fKP-I equation in $X_\frac{\alpha}{2}\cap L^3(\mathbb{R}^2)$ \cite[Proposition 1.2]{linares1}.

\medskip

We would like to mention that there are several existence results on lump solutions using variational approaches for other two-dimensional equations. 
The full water-wave problem admits lump solutions both for strong \cite{gs,bgsw} and weak \cite{bgw} surface tension. In the strong surface tension case the lump solutions can be approximated by rescalings of KP-I lumps, while in the weak surface tension case the lump solutions can be approximated by rescalings of Davey-Stewartson type solitary waves.
The full dispersion KP (FDKP) equation was introduced in \cite[chapter 8]{lannes} as a model for weakly transversal three dimensional water-waves which preserves the dispersion relation of the full water-wave problem. Just as for the classical and fractional KP equation, the FDKP equation can be considered for both strong (FDKP-I) and weak (FDKP-II) weak surface tension. In \cite{eg} it was shown that the FDKP-I equation admits lump solutions and later on in \cite{egn} it was shown that also the FDKP-II equation possesses lump solutions. This is in contrast to the fKP-II equation, which does not admit any lump solutions \cite{linares1}. Just like for the full water-wave problem, in the strong surface tension case the lump solutions can be approximated by rescalings of KP-I lumps, while in the weak surface tension case the lump solutions can be approximated by rescalings of Davey--Stewartson type solitary-waves.

\medskip

\textbf{On the decay result:} The proof of Theorem \ref{thm:decay} on the decay properties of lump solutions is closely related to that of \cite[Theorem 3.1.2]{BonaLi} and \cite[Theorem 4.1]{deBouard}. The steady equation \eqref{eq:steady_KPI} can be rewritten as a convolution equation of the form 
\begin{equation*}
		\phi=\frac{1}{2}K_\alpha*\phi^2,\qquad \hat K_\alpha(\xi_1,\xi_2)= m_\alpha(\xi_1,\xi_2),
	\end{equation*}
where the symbol $m_\alpha$ is given by
\[
 m_\alpha(\xi_1,\xi_2)=\frac{\xi_1^2}{|\xi|^2+\xi_1^{\alpha +2}}.
\]

\begin{remark}\label{rem:optimal}
\emph{An immediate consequence of the discontinuity of the symbol $m_\alpha$ at the origin is that any nontrivial, continuous lump solution of \eqref{eq:steady_KPI} decays \emph{at most} quadratically. Let us assume for a contradiction that $\phi$ is a nontrivial, continuous lump solution, which decays at infinity as $|\cdot|^{-\delta}$ for some $\delta>2$. Then $\phi \in L^1(\R^2)$, which implies that the Fourier  transformation of $\phi$ is continuous. But
$
\hat \phi = \frac{1}{2}m_{\alpha}\hat{\phi^2}
$
cannot be continuous at the origin, since $\hat{\phi^2}(0,0)>0$ and $m_\alpha$ is discontinuous at the origin. We conclude that the singularity of the symbol $m_\alpha$ induced by the transverse direction forces the decay of any nontrivial, continuous lump solution to be at most quadratic. 
}
\end{remark}

\begin{remark}
\emph{In view of Remark \ref{rem:optimal} the decay rate in Theorem \ref{thm:decay} is optimal.}
\end{remark}

The idea is to study the kernel function $K_\alpha$ and to show that it has exactly quadratic decay at infinity (independent of $\alpha$). Then the decay properties of $K_\alpha$ are used to show that also $\phi$ decays quadratically at infinity.

\medskip

\textbf{On the numerics:} We conduct numerical experiments to observe the lump solutions and some of their properties. For this purpose, we generate the solutions numerically  by using Petviashvili iteration method. The method was proposed first by Petviashvili \cite{petviashvili} to compute the lump solutions of the KP-I equation. The convergence of the method for the KP-equation was later discussed in \cite{pelinovski} and now it is widely used to numerically evaluate  traveling wave solutions of evolution equations (see for example \cite{amaral, oruc, pelinovski2} and the references therein).

\medskip

Applying  the Fourier transform to  \eqref{eq:steady_KPI} with respect to the space variables $(x,y)$ we obtain
\begin{equation}\label{eqF}
	c\widehat{\phi}-\frac{1}{2}\widehat{\phi^2}+|\xi_1|^{\alpha} \widehat{\phi}+\frac{\xi_2^2}{\xi_1^2} \widehat{\phi}=0.
\end{equation}
An iterative algorithm for the equation \eqref{eqF} can be proposed as
\begin{equation}\label{iterationv1}
\widehat{\phi}_{n+1}(\xi_1,\xi_2)=\frac{\widehat{\phi^2_n}(\xi_1,\xi_2)} {2(c+\frac{\xi_2^2}{\xi_1^2}+|\xi_1|^{\alpha})} ,~~~~ n=1,2,\dots,
\end{equation}
where $\phi_n$ is the $n^{th}$ iteration of the numerical solution. Since \eqref{iterationv1} is generally divergent the Petviashvili iteration is given as
\begin{equation}\label{iteration}
\widehat{\phi}_{n+1}(\xi_1,\xi_2)=\frac{(M_n)^{\nu}} {2(c+\frac{\xi_2^2}{\xi_1^2}+|\xi_1|^{\alpha})} \widehat{\phi^2_n}(\xi_1,\xi_2),~~~~ n=1,2,\dots,
\end{equation}
by introducing the stabilizing factor
\begin{equation*}
  M_n=\frac  {  \int_{\mathbb{R}^2} 2 (c+\frac{\xi_2^2}{\xi_1^2}+|\xi_1|^{\alpha})~(\widehat{\phi_n})^2   \mathrm{d}(\xi_1,\xi_2)  }
             {  \int_{\mathbb{R}^2}  \widehat{\phi^2_n}~\widehat{\phi_n}~ \mathrm{d}(\xi_1,\xi_2)   }.
\end{equation*}
Here the free parameter $\nu$ is chosen as $2$ for the fastest convergence. To evaluate the term $1/ \xi_1^2 $ for $\xi_1=0$,  we regularize it as $1/( \xi_1+i\lambda)^2 $, where $ \lambda = 2.2 \times 10^{-16}$ as in \cite{klein1, klein2}. We  control the iterative  process by the error between two consecutive iterations
\begin{equation*}
    \mbox{error}(n)=\| \phi_n-\phi_{n-1}\|_\infty,~~~~ n=1,2,\dots,
\end{equation*}
by the stabilization factor error $|1-M_n|$, and the residual error
\begin{equation*}
    Res(n)=\| \mathcal{S}\phi_n\|_\infty,~~~~ n=1,2,\dots
\end{equation*}
where 
\begin{equation*}
    \mathcal{S}=\left(-c\phi +\frac{1}{2}\phi^2-\mathrm{D}^\alpha_x \phi\right)_{xx}-\phi_{yy}.
\end{equation*}
We make sure that  the errors are of order less than $10^{-5} $. In addition, we control the decay of Fourier coefficients $\hat{\phi}(\xi_1, \xi_2)$ in the numerical experiments. 



\medspace

\medskip

\subsection{Notation and organization of the paper}
We first introduce a notation, which is frequently used in the sequel. Let $f$ and $g$ be two { positive} functions.  We write $f\lesssim g$ ($f\gtrsim g$) if there exists a constant $c>0$ such that $f\leq c g$ ($f\geq cg$). Moreover, we use the notation $f\eqsim g$ whenever $f\lesssim g$ and $f\gtrsim g$. 

\medskip

We conclude the introduction by the organization of the paper:  In Section \ref{existence_sec} we prove existence of lump solutions for the fKP-I equation (Theorem \ref{existence_theorem}) via a variational approach. We also present numerically generated lump solutions and observe the cross-sectional symmetry of the solutions numerically. Section \ref{decay_sec} is devoted to the proof of Theorem \ref{thm:decay}, which relies upon a careful study of the decay and regularity of the kernel function $K_\alpha$. The appendix contains some technical results which are needed for the analysis in Section \ref{decay_sec}.

\bigskip
	\section{Existence of solitary solutions}\label{existence_sec}
 
We consider the (rescaled) traveling wave fKP-I equation:
\begin{equation}\label{travelling_eq_sec_existence}
\phi+\mathrm{D}_x^\alpha \phi+\partial_x^{-2}\partial_y^2\phi-\frac{\phi^2}{2}=0.
\end{equation}
Equation \eqref{travelling_eq_sec_existence} can be realized as a constrained minimization problem. Indeed, let
\begin{align*}
\mathcal{L}(\phi)=\frac{1}{2}\int_{\mathbb{R}^2} \left(\phi^2+(\mathrm{D}_x^\frac{\alpha}{2}\phi)^2+(\partial_x^{-1}\partial_y\phi)^2 \ \right)\mathrm{d}(x,y),\qquad \mathcal{N}(\phi)=\frac{1}{6}\int_{\mathbb{R}^2}\phi^3\ \mathrm{d}(x,y),
\end{align*}
which we study in the space $X_\frac{\alpha}{2}$ 
and consider the constrained minimization problem 
\begin{equation}\label{min_problem}
I_\mu=\inf\{\mathcal{L}(\phi)\colon \phi\in X_\frac{\alpha}{2},\ \mathcal{N}(\phi)=\mu\}.
\end{equation}
In order to find nontrivial solutions we assume that $\mu\neq 0$ and without loss of generality we may further assume that $\mu > 0$.

Let $\phi$ be a solution of \eqref{min_problem}. Then there exists a Lagrange multiplier $\lambda\in \R$ such that 
\begin{equation}\label{lagrange1}
\mathrm{d}\mathcal{L}(\phi)-\lambda\mathrm{d}\mathcal{N}(\phi)=0.
\end{equation}
Since
\begin{align*}
\mathrm{d}\mathcal{L}(\phi)=\phi+\mathrm{D}_x^\alpha\phi+\partial_x^{-2}\partial_y^2\phi,\qquad
\mathrm{d}\mathcal{N}(\phi)=\frac{1}{2}\phi^2.
\end{align*}
equation \eqref{lagrange1} becomes
\begin{equation*}
\phi+\mathrm{D}_x^\alpha\phi+\partial_x^{-2}\partial_y^2\phi-\lambda\frac{\phi^2}{2}=0.
\end{equation*}
By rescaling
\begin{equation*}
\phi(x,y)=|\lambda|^{-1}\tilde{\phi}(x,y),
\end{equation*}
we find that $\tilde{\phi}$ satisfies the equation 
\begin{equation*}
\tilde{\phi}+\mathrm{D}_x^\alpha\tilde{\phi}+\partial_x^{-2}\partial_y^2\tilde{\phi}-\text{sgn}(\lambda)\frac{\tilde{\phi}^2}{2}=0.
\end{equation*}
If $\lambda>0$, this is equation \eqref{travelling_eq_sec_existence} and if $\lambda<0$ we apply the transformation $\tilde{\phi}\mapsto -\tilde{\phi}$ and again recover \eqref{travelling_eq_sec_existence}. Therefore, in order to prove the existence of the solutions of equation \eqref{travelling_eq_sec_existence}, we will prove existence of solutions of the constrained minimization problem \eqref{min_problem}.

\medskip

In the sequel, let us fix $\mu>0$ (this will ensure that $I_\mu>0$, see Corollary \ref{positivity}) and let $\{\phi_n\}_{n\in \N}\subset X_{\frac{\alpha}{2}}$ be a minimizing sequence such that $\mathcal{N}(\phi_n)=\mu$ and $\lim_{n\to \infty}\mathcal{L}(\phi_n)=I_\mu$. We aim to show that there exists a subsequence (not relabeled) of $\{\phi_n\}_{n\in \N}$, which converges to a function $\phi \in X_{\frac{\alpha}{2}}$ satisfying $\mathcal{L}(\phi)=I_\mu$ and $\mathcal{N}(\phi)=\mu$.

Let us set
\begin{equation*}
e_n=\frac{1}{2}\left(\phi_n^2+(\mathrm{D}_x^\frac{\alpha}{2}\phi_n)^2+(\partial_x^{-1}\partial_y\phi_n)^2\right)
\end{equation*}
and note that
\begin{equation*}
\mathcal{L}(\phi_n)=\int_{\mathbb{R}^2}e_n\ \mathrm{d}(x,y).
\end{equation*}
We will use the following version of the concentration--compactness theorem for the sequence $\{e_n\}_{n\in \N}$ and show that the concentration scenario occurs. This is then used to construct a convergent subsequence of $\{\phi_n\}_{n\in \N}$, converging to a solution of \eqref{min_problem}

\begin{theorem}\label{cc}
Let $d\in \N$. Any sequence $\{e_n\}_{n\in \N}\subset L^1(\mathbb{R}^d)$ of non-negative functions such that 
\begin{equation*}
\lim_{n\rightarrow \infty}\int_{\mathbb{R}^d}e_n\ \mathrm{d}x=I>0,
\end{equation*}
admits a subsequence, denoted again by $\{e_n\}_{n\in\mathbb{N}}$, for which one of the following phenomena occurs:
\begin{itemize}
\item \textbf{Vanishing:} For each $r>0$, one has
\begin{equation*}
\lim_{n\rightarrow \infty}\left(\sup_{x\in \mathbb{R}^d}\int_{B_r(x)}e_n\ \mathrm{d}x\right)=0.
\end{equation*}
\item \textbf{Dichotomy:} There are sequences $\{x_n\}_{n\in\mathbb{N}}\subset \mathbb{R}^d$, $\{M_n\}_{n\in\mathbb{N}}, \{N_n\}_{n\in\mathbb{N}}\subset \mathbb{R}$ and $I^*\in(0,I)$ such that $M_n, N_n \rightarrow \infty,\ \frac{M_n}{N_n}\rightarrow 0$ and 
\begin{equation*}
\lim_{n\rightarrow \infty}\int_{B_{M_n}(x_n)}e_n\ \mathrm{d}x=I^*,\quad \lim_{n\rightarrow \infty}\int_{B_{N_n}(x_n)}e_n\ \mathrm{d}x=I^*.
\end{equation*} 
\item \textbf{Concentration:} There exists a sequence $\{x_n\}_{n\in\mathbb{N}}\subset \mathbb{R}^d$ with the property that for each $\varepsilon>0$, there exists $r>0$ with 
\begin{equation*}
\int_{B_r(x_n)}e_n\ \mathrm{d}x\geq I-\varepsilon,~ \text{ for all }~n\in\mathbb{N}.
\end{equation*} 
\end{itemize}
\end{theorem}
Interpreting $I$ as a mass, Theorem \ref{cc} says that $\{e_n\}_{n\in\mathbb{N}}$ admits a subsequence for which one of the following occur: The the mass spreads out in $\mathbb{R}^n$ (vanishing), it splits into two parts (dichotomy) or the mass is uniformly concentrated in $\mathbb{R}^n$ (concentration).
\medskip

\subsection{Preliminary results}
In this subsection we will gather some of the results we need in order to apply Theorem \ref{cc}.
\begin{proposition}\label{technical_results}
Let $\phi\in X_\frac{\alpha}{2}$. Then, 
\begin{itemize}
\item[(i)] $\mathcal{L}(\phi)=\frac{1}{2}\norm{\phi}_\frac{\alpha}{2}^2$,
\item[(ii)] for $\frac{4}{5}\leq \alpha \leq 2$, one has the anisotropic Sobolev inequality 
\begin{align*}
\norm{\phi}_{L^3(\R^2)}^3\lesssim \norm{\phi}_{L^2(\R^2)}^\frac{5\alpha-4}{\alpha+2}\norm{\mathrm{D}_x^\frac{\alpha}{2}\phi}_{L^2(\R^2)}^\frac{18-5\alpha}{2(\alpha+2)}\norm{\partial_x^{-1}\partial_y\phi}_{L^2(\R^2)}^\frac{1}{2}.
\end{align*}
In particular $X_{\frac{\alpha}{2}}\subset L^3(\R^2)$ and $\norm{\phi}_{L^3(\R^2)}\lesssim\norm{\phi}_\frac{\alpha}{2}$ for all $\alpha\geq \tfrac{4}{5}$.
\end{itemize}
\end{proposition}
\begin{proof}
Part $(i)$ is immediate while part $(ii)$ can be found in \cite[Lemma 1.1]{linares1}.
\end{proof}
\begin{corollary}\label{positivity}
The minimum $I_\mu$ is positive.
\end{corollary}
\begin{proof}
By Proposition \ref{technical_results} we have that 
\begin{align*}
\mu=\mathcal{N}(\phi)\lesssim \norm{\phi}_{L^3}^3\lesssim   \norm{\phi}_\frac{\alpha}{2}^3\eqsim  \mathcal{L}(\phi)^\frac{3}{2}.
\end{align*}
\end{proof}
Corollary \ref{positivity} ensures that the minimizer is not given by the trivial solution.
\begin{lemma}\label{embedding_lemma}
For any $\alpha>0$, the space $X_\frac{\alpha}{2}$ is compactly embedded in $L_{\text{loc}}^2(\mathbb{R}^2)$.
\end{lemma}
\begin{proof}
The proof follows essentially the lines in \cite[Lemma 3.3]{deBouard}. We include it here for the sake of completeness.

For $\phi\in X_\frac{\alpha}{2}$, let $\varphi=\partial_x^{-1}\phi$. From the definition of $X_\frac{\alpha}{2}$ we find that $\partial_x\varphi,\ \partial_y\varphi\in L^2(\mathbb{R}^2)$, that is, $\varphi\in \dot{H}^1(\mathbb{R}^2)$. 
From Poincare's inequality we have that $\dot{H}^1(\mathbb{R}^2)$ is continuously embedded in $\text{BMO}(\mathbb{R}^2)$. It follows from this that $\varphi\in \text{BMO}(\mathbb{R}^2)\subset L_{\text{loc}}^q(\mathbb{R}^2)$ for all $0<q<\infty$.
Let $\{\phi_n\}_{n=1}^\infty$ be a bounded sequence in $X_\frac{\alpha}{2}$. We will show that for any $R>0$ there exists a subsequence $\{\phi_{n_k}\}_{n=1}^\infty$, which converges in $L^2(B_R)$, where $B_R$ is the ball of radius $R$ centered at the origin in $\mathbb{R}^2$. Let $\varphi_n=\partial_x^{-1}\phi_n$. Since we are only interested in convergence in $L^2(B_R)$, we may assume that $\varphi_n$ is supported on $B_{2R}$ by multiplying $\varphi_n$ with a smooth cutoff function $\psi$ such that $\psi\equiv 1$ in $B_R$ and $\text{supp}(\psi)\subset B_{2R}$. It follows then that $\phi_n$ is supported on $B_{2R}$ as well. 

Since $\{\phi_n\}_{n=1}^\infty$ is bounded in $X_\frac{\alpha}{2}$ we can extract a subsequence, which we still denote by $\{\phi_n\}_{n=1}^\infty$, such that $\phi_n\rightharpoonup \phi$, for some $\phi\in X_\frac{\alpha}{2}$. Moreover, by replacing $\phi_n$ with $\phi_n-\phi$, we may assume that $\phi=0$. Our aim is then to show that 
\begin{equation*}
\int_{\mathbb{R}^2}|\phi_n|^2\ \mathrm{d}(x,y)\rightarrow 0,\ ~\text{as }~n\rightarrow \infty.
\end{equation*}

Let $R_1>0$. We have
\begin{align}
\int_{\mathbb{R}^2}|\phi_n|^2\ \mathrm{d}(x,y)&=\int_{\mathbb{R}^2}|\hat{\phi}_n|^2\ \mathrm{d}(\xi_1,\xi_2)\nonumber\\
&=\int_{\{|\xi_1|\leq R_1,\ |\xi_2|\leq R_1^2\}}|\hat{\phi}_n|^2\ \mathrm{d}(\xi_1,\xi_2)+\int_{\{|\xi_1|\geq R_1\}}|\hat{\phi}_n|^2\ \mathrm{d}(\xi_1,\xi_2)\nonumber\\
&\quad +\int_{\{|\xi_1|\leq R_1,\ |\xi_2|\geq R_1^2\}}|\hat{\phi}_n|^2\ \mathrm{d}(\xi_1,\xi_2).\label{embedding_integrals}
\end{align}
We proceed to estimate each integral on the right-hand side of \eqref{embedding_integrals} separately. For the third integral we can write
\begin{align*}
\int_{\{|\xi_1|\leq R_1,\ |\xi_2|\geq R_1^2\}}|\hat{\phi}_n|^2\ \mathrm{d}(\xi_1,\xi_2)&=\int_{\{|\xi_1|\leq R_1,\ |\xi_2|\geq R_1^2\}}\frac{\xi_1^2}{\xi_2^2}|\mathcal{F}(\partial_x^{-1}\partial_y\phi_n)|^2\ \mathrm{d}(\xi_1,\xi_2)\\
&\leq \frac{R_1^2}{R_1^4}\norm{\partial_x^{-1}\partial_y\phi_n}_{L^2(\R^2)}^2\\
&=\frac{1}{R_1^2}\norm{\partial_x^{-1}\partial_y\phi_n}_{L^2(\R^2)}^2
\end{align*}
and for the second one
\begin{align*}
\int_{\{|\xi_1|\geq R_1\}}|\hat{\phi}_n|^2\ \mathrm{d}(\xi_1,\xi_2)&=\int_{\{|\xi_1|\geq R_1\}}\frac{1}{|\xi_1|^\alpha}|\mathcal{F}(\mathrm{D}_x^\frac{\alpha}{2}\phi_n)|^2\ \mathrm{d}(\xi_1,\xi_2)\leq \frac{1}{R_1^\alpha}\norm{\mathrm{D}_x^\frac{\alpha}{2}\phi_n}_{L^2(\R^2)}^2.
\end{align*}
From these estimates we conclude that, given $\varepsilon>0$ we can choose $R_1$ sufficiently large such that
\begin{equation*}
\int_{\{|\xi_1|\geq R_1\}}|\hat{\phi}_n|^2\ \mathrm{d}(\xi_1,\xi_2)+\int_{\{|\xi_1|\leq R_1,\ |\xi_2|\geq R_1^2\}}|\hat{\phi}_n|^2\ \mathrm{d}(\xi_1,\xi_2)<\varepsilon.
\end{equation*}
In order to deal with the first integral, we first note that since $\phi_n\rightharpoonup 0$ in $X_\frac{\alpha}{2}$, we have 
\begin{equation*}
\hat{\phi}_n(\xi_1,\xi_2)=\int_{B_{2R}}\mathrm{e}^{-\mathrm{i}(x\xi_1+y\xi_2)}\phi_n(x,y)\ \mathrm{d}(x,y)\rightarrow 0\ ~\text{as }~n\rightarrow \infty.
\end{equation*}
Moreover, 
\begin{equation*}
|\hat{\phi}_n(\xi_1,\xi_2)|\leq \norm{\phi_n}_{L^1(B_{2R})}\lesssim \norm{\phi_n}_{L^2(B_{2R})}.
\end{equation*}
Since $\{\phi_n\}_{n=1}^\infty$ is bounded in $X_\frac{\alpha}{2}$ we can conclude that $\{\hat{\phi}_n\}_{n=1}^\infty$ is bounded in $L^\infty(\mathbb{R}^2)$, so by the dominated convergence theorem
\begin{equation*}
\int_{\{|\xi_1|\leq R_1,\ |\xi_2|\leq R_1^2\}}|\hat{\phi}_n|^2\ \mathrm{d}(\xi_1,\xi_2)\rightarrow 0,\ \text{as }n\rightarrow \infty.
\end{equation*} 
\end{proof}

Next we prove that $I_\mu$ is subadditive as a function of $\mu$, a property which will be crucial when proving that the dichotomy scenario in Theorem \ref{cc} does not occur. 
\begin{proposition}\label{subadditivity}
The infimum $I_\mu$ is strictly increasing and subadditive as a function of $\mu$, that is
\begin{equation*}
I_{\mu_1+\mu_2}<I_{\mu_1}+I_{\mu_2},~\text{ for all }~\mu_1,\mu_2>0.
\end{equation*}
\end{proposition}
\begin{proof}
Let $h\in X_\frac{\alpha}{2}$ be such that $\mathcal{N}(h)=1$ and let $\phi=\mu^\frac{1}{3}h$. Then $N(\phi)=\mu$ and $\mathcal{L}(\phi)=\mu^\frac{2}{3}\mathcal{L}(h)$, which implies
\begin{equation*}
I_\mu=\mu^\frac{2}{3}I_1,
\end{equation*}
from which the statement in the proposition directly follows.
\end{proof}
When applying Theorem \ref{cc} we will be taking integrals over bounded domains. It is therefore useful to consider the norm $\norm{\cdot}_\frac{\alpha}{2}$ restricted to a bounded domain $\Omega\subset\mathbb{R}^2$:
\begin{equation*}
||\phi||_{\frac{\alpha}{2},\Omega}^2=\norm{\phi}_{L^2(\Omega)}^2+\norm{\mathrm{D}_x^\frac{\alpha}{2}\phi}_{L^2(\Omega)}^2+\norm{\partial_x^{-1}\partial_y\phi}_{L^2(\Omega)}^2.
\end{equation*}
We also make the following definition.
\begin{definition}
Let $\Omega$ be a bounded domain in $\mathbb{R}^2$. For $f\in L_{\text{loc}}^1(\mathbb{R}^2)$, let \[
f_\Omega:=f-M_\Omega(f),\] where
$
M_\Omega(f):=\frac{1}{|\Omega|}\int_\Omega f \ \mathrm{d}(x,y)
$ is the mean of $f$ over $\Omega$.
\end{definition}
When proving that the vanishing scenario does not occur we will make use of the following result.

\begin{proposition}\label{localization}
Let $\phi\in X_\frac{\alpha}{2}$, $\varphi=\partial_x^{-1}\phi$ and let $\psi$ be a smooth cutoff function supported on a bounded domain\footnote{The proposition can be generalized to domains, which are given by disjoint unions of type $\Omega$.} $\Omega=\{(x,y)\in \R^2 \mid y\in (a,b), x\in (h_1(y),h_2(y))\}$, for some $a,b\in \R$ and $h_i\in C([a,b])$, $i=1,2$. Define 
\begin{equation*}
F_\Omega(\phi)=\partial_x(\psi(\varphi_\Omega)).
\end{equation*}
Then,
\begin{equation*}
\norm{F_\Omega(\phi)}_{\frac{\alpha}{2}}\lesssim \norm{\phi}_{\frac{\alpha}{2},\Omega}.
\end{equation*}
\end{proposition}

\begin{proof}
We have that 
\begin{equation*}
\norm{F_\Omega(\phi)}_{\frac{\alpha}{2}}^2=\norm{\partial_x(\psi\varphi_\Omega)}_{L^2(\R^2)}^2+\norm{\partial_y(\psi\varphi_\Omega)}_{L^2(\R^2)}^2+\norm{\mathrm{D}_x^\frac{\alpha}{2}\partial_x(\psi\varphi_\Omega)}_{L^2(\R^2)}^2.
\end{equation*}
We consider each of these terms separately. The first term can be estimated as
\begin{equation*}
\norm{\partial_x(\psi\varphi_\Omega)}_{L^2(\R^2)}=\norm{\psi_x\varphi_\Omega+\psi\phi}_{L^2(\R^2)}\leq \norm{\psi_x\varphi_\Omega}_{L^2(\R^2)}+\norm{\psi\phi}_{L^2(\R^2)},
\end{equation*}
and $\norm{\psi\phi}_{L^2(\R^2)}\lesssim \norm{\phi}_{L^2(\Omega)}$, while
\begin{align*}
\norm{\psi_x\varphi_\Omega}_{L^2(\R^2)}\lesssim \norm{\varphi_\Omega}_{L^2(\Omega)}\lesssim \norm{\varphi_x}_{L^2(\Omega)}+\norm{\varphi_y}_{L^2(\Omega)}=\norm{\phi}_{L^2(\Omega)}+\norm{\partial_x^{-1}\partial_y\phi}_{L^2(\Omega)},
\end{align*}
where we used Poincaré's inequality and the definition $\varphi = \partial_x^{-1}\phi$.
Hence,
\begin{equation}\label{eq:11}
\norm{\partial_x(\psi\varphi_\Omega)}_{L^2}^2\lesssim \norm{\phi}_{L^2(\Omega)}+\norm{\partial_x^{-1}\partial_y\phi}_{L^2(\Omega)}
\end{equation}
and in the same way we find  
\begin{equation}\label{eq:12}
\norm{\partial_y(\psi\varphi_\Omega)}_{L^2}^2\lesssim \norm{\phi}_{L^2(\Omega)}+\norm{\partial_x^{-1}\partial_y\phi}_{L^2(\Omega)}.
\end{equation}
Recall that $\Omega=\{(x,y)\in \R^2 \mid y\in (a,b), x\in (h_1(y),h_2(y))\}$ and
set $\Omega_y:=(h_1(y),h_2(y))$. By the Leibniz' rule for fractional derivatives (see e.g. \cite[Theorem 7.6.1]{Grafakos}), we can estimate
\begin{align*}
\norm{\mathrm{D}_x^\frac{\alpha}{2}\partial_x(\psi\varphi_\Omega)}_{L^2(\R^2)}^2 &= \int_{\R} \norm{\mathrm{D}_x^\frac{\alpha}{2}\partial_x(\psi\varphi_\Omega)(\cdot, y)}_{L^2(\Omega_y)}^2\, \mathrm{d}y\\
&\lesssim \int_{a}^b \norm{\mathrm{D}_x^\frac{\alpha}{2}\partial_x \psi(\cdot, y)}_{L^\infty(\Omega_y)}^2\norm{\varphi_\Omega(\cdot, y)}_{L^2(\Omega_y)}^2+ \norm{\psi(\cdot, y)}_{L^\infty(\Omega_y)}^2 \norm{\mathrm{D}_x^\frac{\alpha}{2}\partial_x \varphi_\Omega(\cdot, y)}_{L^2(\Omega_y)}^2 \mathrm{d}y.
\end{align*}
Using that $\psi$ is a smooth function, we conclude by Poincaré's inequality that 
\begin{align}\label{eq:13}
\begin{split}
\norm{\mathrm{D}_x^\frac{\alpha}{2}\partial_x(\psi\varphi_\Omega)}_{L^2(\R^2)}^2 &\lesssim \norm {\varphi_{\Omega}}_{L^2(\Omega)}^2 + \norm{\mathrm{D}_x^\frac{\alpha}{2}\partial_x \varphi_\Omega}_{L^2(\Omega)}^2\\ 
&\lesssim \norm{\varphi_x}_{L^2(\Omega)}+\norm{\varphi_y}_{L^2(\Omega)}+ \norm{\mathrm{D}_x^\frac{\alpha}{2}\partial_x \varphi}_{L^2(\Omega)}^2 \\
&=\norm{\phi}_{L^2(\Omega)}+\norm{\partial_x^{-1}\partial_y\phi}_{L^2(\Omega)}+ \norm{\mathrm{D}_x^\frac{\alpha}{2}\phi}_{L^2(\Omega)}^2
\end{split}
\end{align}
Gathering \eqref{eq:11}, \eqref{eq:12}, and \eqref{eq:13}, we have shown that 
\begin{equation*}
\norm{F_\Omega(\phi)}_{\frac{\alpha}{2}}\lesssim \norm{\phi}_{\frac{\alpha}{2},\Omega}.
\end{equation*}

\end{proof}

Eventually, when excluding the dichotomy scenario we will make use of the following lemma, which provides a Poincaré-like inequality. 
\begin{lemma}[\cite{deBouard0}, Lemma 3.1]\label{deBourd_Saut_lemma}

Let $2\leq p<\infty $ and $R>0$. Then there exists a positive constant $C$ such that for all $f\in L_\text{loc}^1(\mathbb{R}^2)$ one has that
\[
\|f_{A_{2R,R}}\|_{L^p(A_{2R,R})} \leq C R^\frac{2}{p}\|\nabla f\|_{L^2(A_{2R,R})},
\]
where $A_{2R,R}\subset \R^2$ denotes the annulus centered at the origin of radii $2R>R$. 
\end{lemma}

\subsection{Existence of minimizers}
Let $\{\phi_n\}_{n\in\mathbb{N}}\subset X_\frac{\alpha}{2}$ be a minimizing sequence for the constrained minimization problem \eqref{min_problem}, that is, $\mathcal{N}(\phi_n)=\mu$ and $\lim_{n\rightarrow \infty} \mathcal{L}(\phi_n)=I_\mu$. We will apply Theorem \ref{cc} to the sequence 
\begin{equation*}
e_n=\frac{1}{2}\left(\phi_n^2+(\mathrm{D}_x^\frac{\alpha}{2}\phi_n)^2+(\partial_x^{-1}\partial_y\phi_n)^2\right).
\end{equation*}
Recall that
\begin{equation*}
\int_{\mathbb{R}^2}e_n\ \mathrm{d}(x,y)=\mathcal{L}(\phi_n).
\end{equation*}
We will show in Proposition \ref{prop:vanishing} and Proposition \ref{prop:dichotomy} that the vanishing and dichotomy scenarios in Theorem \ref{cc} does not occur and then use the concentration scenario to construct a convergent subsequence of $\{\phi_n\}_{n\in\mathbb{N}}$, converging to a solution $\phi$ of \eqref{min_problem}. 

\medskip

\begin{proposition}[Excluding "vanishing"] \label{prop:vanishing}
No subsequence of $\{e_n\}_{n\in\mathbb{N}}$ has the \emph{vanishing} property in Theorem \ref{cc}.
\end{proposition}

\begin{proof}
Assume for a contradiction that vanishing does occur, that is
\begin{equation*}
\lim_{n\rightarrow\infty}\left(\sup_{(x,y)\in\mathbb{R}^2}\int_{B_r(x,y)}e_n\ \mathrm{d}(x,y)\right)=0
\end{equation*}
for each $r>0$. Let us cover $\mathbb{R}^2$ with balls $B_{1,j}$, $j\in\mathbb{N}$, of radius $1$ such that each point in $\mathbb{R}^2$ is contained in at most three balls. Let $\{\psi_j\}_{n\in\mathbb{N}}$ be a smooth partition of unity such that $\text{supp}(\psi_j)\subset B_{1,j}$. Using Proposition \ref{technical_results} (ii) and Proposition \ref{localization} we find 
\begin{align*}
|\mathcal{N}(\phi_n)|&\lesssim \norm{\phi_n}_{L^3(\R^2)}^3\\
&\lesssim \sum_{j\in\mathbb{N}}\norm{F_{B_{1,j}}(\phi_n)}_{L^3(\R^2)}^3\\
&\leq\sup_{j\in\mathbb{N}}\norm{F_{B_{1,j}}(\phi_n)}_{L^3(\R^2)}\sum_{j\in\mathbb{N}}\norm{F_{B_{1,j}}(\phi_n)}_{L^3(\R^2)}^2\\
&\lesssim \sup_{j\in\mathbb{N}}\norm{F_{B_{1,j}}(\phi_n)}_{\frac{\alpha}{2}}\sum_{j\in\mathbb{N}}\norm{F_{B_{1,j}}(\phi_n)}_{\frac{\alpha}{2}}^2\\
&\lesssim \sup_{j\in\mathbb{N}}\norm{\phi_n}_{\frac{\alpha}{2},B_{1,j}}\sum_{j\in\mathbb{N}}\norm{\phi_n}_{\frac{\alpha}{2}, B_{1,j}}^2\\
&\lesssim \sup_{j\in\mathbb{N}}\left(\int_{B_{1,j}}e_n\ \mathrm{d}(x,y)\right)^\frac{1}{2}\norm{\phi_n}_\frac{\alpha}{2}^2.
\end{align*}
By letting $n\rightarrow \infty$ we get  $\mathcal{N}(\phi_n)\rightarrow 0$, which contradicts the fact that $\mathcal{N}(\phi_n)=\mu>0$.
\end{proof}

\begin{proposition}[Excluding "dichotomy"] \label{prop:dichotomy}
No subsequence of $\{e_n\}_{n\in\mathbb{N}}$ has the \emph{dichotomy} property in Theorem \ref{cc}.
\end{proposition}
\begin{proof}
Throughout the proof we will use $B_R$ to denote the ball in $\mathbb{R}^2$ centered at the origin of radius $R>0$ and $A_{R_1,R_2}$ to denote the annulus centered at the origin of radii $R_1>R_2>0$.
\medskip

Assume for a contradiction that the dichotomy scenario in Theorem \ref{cc} occurs, that is there exist sequences $\{(x_n,y_n)\}_{n\in \N}\subset \R^2, \{M_n\}_{n\in \N}, \{N_n\}_{n\in \N}\subset \R$ and $I^* \in (0,I_\mu)$ with $M_n, N_n, \frac{N_n}{M_n} \to \infty$ for $n\to \infty$ and 
\begin{equation}\label{eq:d}
\lim_{n\rightarrow \infty}\int_{B_{M_n}(x_n,y_n)}e_n\ \mathrm{d}x=I^*,\quad \lim_{n\rightarrow \infty}\int_{B_{N_n}(x_n,y_n)}e_n\ \mathrm{d}x=I^*.
\end{equation} 
We will show that this leads to a contradiction, by proving that provided \eqref{eq:d} holds, there exists two sequences $\{\omega_n^{(1)}\}_{n\in \N}, \{\omega_n^{(2)}\}_{n\in \N}$, which have in the limit $n\to \infty$ disjoint support and
\begin{itemize}
\item[(i)]$ \mathcal{N}(\omega_n^{(1)}) + \mathcal{N}(\omega_n^{(2)}) -\mathcal{N}(\omega_n)\to 0$,
\item[(ii)] $\mathcal{L}(\omega_n^{(1)})\to I^*$ and $\mathcal{L}(\omega_n^{(2)})\to (I_\mu-I^*)$,
\end{itemize}
where $\omega_n=\phi_n(\cdot + (x_n,y_n))$ is the shift of $\phi_n$ by $(x_n,y_n)$. We shift the function $\phi_n$ for reasons of convenience in order to work with balls and annuli centered at the origin instead of at $(x_n,y_n)$. Notice that if (i) and (ii) hold we obtain a contradiction due to the subadditivity of the $I_\mu$ stated in Proposition \ref{subadditivity}: Set
\[
\mu_{1,n}:=  \mathcal{N}(\omega_n^{(1)})\qquad \mbox{and}\qquad \mu_{2,n}:=  \mathcal{N}(\omega_n^{(2)}),
\]
and $\mu_i:=\lim_{n\to \infty}\mu_{i,n}$ for $i=1,2$. Then (i) implies that $\mu_1+\mu_2=\mu$, since $\mathcal{N}(\omega_n)=\mu$ for all $n\in \N$. First we show that  $\mu_1\neq 0$. If $\mu_1= 0$, then $\mu_2= \mu$. By setting
\[
\tilde \omega_n^{(2)}:= \left(\frac{\mu}{\mu_{2,n}}\right)^\frac{1}{3}\omega_n^{(2)}
\]
we find  $\mathcal{N}(\tilde \omega_n^{(2)})=\mu$ for all $n\in \N$ and
\[
\left|\mathcal{L}(\tilde \omega_n^{(2)})-\mathcal{L}( \omega_n^{(2)})\right|\to 0 \qquad \mbox{for}\quad n\to \infty,
\]
since $\lim_{n\to \infty} \frac{\mu}{\mu_{2,n}}=1$.
But then by using (ii) we obtain
\[
I_\mu\leq \mathcal{L}(\tilde \omega_n^{(2)}) \to I_\mu-I^* < I_\mu\qquad \mbox{for}\quad n\to \infty,
\]
which is a contradiction. Hence, $\mu_1\neq 0$ and similarly
we find 
$\mu_2\neq 0$. Thus, $|\mu_i|>0$ for $i=1,2$ and we can define the rescaled functions
$$
\bar \omega_n^{(i)}:= \left(\frac{|\mu_i|}{\mu_{i,n}}\right)^\frac{1}{3}\omega_n^{(i)}\qquad \mbox{for}\quad i=1,2,
$$
which satisfy
$
\mathcal{N}(\bar \omega_n^{(i)})=|\mu_i|
$
for all $n\in \N$
and 
\begin{equation*}
\lim_{n\to \infty} \mathcal{L}(\bar \omega_n^{(1)})=I^*,\qquad 
\lim_{n\to \infty} \mathcal{L}(\bar \omega_n^{(1)})=I_\mu-I^*,
\end{equation*}
by (ii) together with $\lim_{n\to \infty} \left|\frac{|\mu_i|}{\mu_{i,n}}\right|=1$. 
Combining this with the subadditivity of $I_\mu$ for $\mu>0$, which is stated in Proposition \ref{subadditivity}, we find the contradiction
\[
I_\mu\leq I_{|\mu_1|+|\mu_2|}<I_{|\mu_1|}+I_{|\mu_2|}\leq \lim_{n\to \infty} \left(\mathcal{L}(\bar \omega_n^{(1)})+\mathcal{L}(\bar \omega_n^{(2)})\right)= I_\mu.
\]

\medskip

We are left to show that there exists two sequences $\{\omega_n^{(1)}\}_{n\in \N}, \{\omega_n^{(2)}\}_{n\in \N}$, which have in the limit $n\to \infty$ disjoint support and satisfy (i), (ii). To this end,
let $\varphi_n=\partial_x^{-1}\phi_n$ and let $\chi\colon \mathbb{R}^2\rightarrow [0,1]$ be a smooth cutoff function such that $\chi(x,y)=1$ for $|(x,y)|\leq 1$ and $\chi(x,y)=0$ for $|(x,y)|\geq 2$. Next let $\sigma_n:=\varphi_n(\cdot +(x_n,y_n))$ and
\begin{equation*}
\sigma_n^{(1)}:=\chi_{1n}\sigma_{n, A_{2M_n, M_n}},\quad \sigma_n^{(2)}:=\chi_{2n}\sigma_{n,A_{N_n,N_n/2}},
\end{equation*}
where 
\begin{equation*}
\chi_{1n}(x,y):=\chi\left(\frac{1}{M_n}(x,y)\right),\quad \chi_{2n}:=1-\chi\left(\frac{2}{N_n}(x,y)\right).
\end{equation*}
Eventually, we define
\begin{equation*}
\omega_n:=\partial_x\sigma_n,\quad \omega_n^{(i)}:=\partial_x\sigma_n^{(i)},\ i=1,2.
\end{equation*}
We remark that by definition $\omega_n=\phi_n(\cdot + (x_n,y_n))$. Furthermore,
\begin{align*}
\text{supp}(\omega_n^{(1)})\subset B_{2M_n}\qquad \mbox{and} \qquad 
\text{supp}(\omega_n^{(2)})\subset \mathbb{R}^2\backslash B_{\frac{N_n}{2}}.
\end{align*}
See Figure \ref{fig:F} for an illustration of the supports for $\omega_n^{(i)}, i=1,2$.

\begin{center}
\begin{figure}[H]
\centering
\begin{tikzpicture}[scale=0.9]
\footnotesize
\draw[->] (-8,0)--(8.1,0) node[below]{$r=|(x,y)|$};
\draw[-] (0,0.1)--(0,-0.1) node[below]{$0$};
\draw[-] (1,0.1)--(1,-0.1) node[below]{$M_n$};
\draw[-] (2,0.1)--(2,-0.1) node[below]{$2M_n$};
\draw[-] (3.2,0.1)--(3.2,-0.1) node[below]{$N_n/2$};
\draw[-] (6.4,0.1)--(6.4,-0.1) node[below]{$N_n$};
\draw[-] (-1,0.1)--(-1,-0.1) node[below]{$-M_n$};
\draw[-] (-2,0.1)--(-2,-0.1) node[below]{$-2M_n$};
\draw[-] (-3.2,0.1)--(-3.2,-0.1) node[below]{$-N_n/2$};
\draw[-] (-6.4,0.1)--(-6.4,-0.1) node[below]{$-N_n$};
\draw[-,line width=0.1cm, orange!40] (-2,0.1)--(2,0.1);
\draw[-,line width=0.1cm, orange] (-1,0.1)--(1,0.1);
\node[orange] at (0,0.5) {$\omega_n^{(1)}=\omega_n$};
\draw[-,line width=0.1cm, blue!40] (3.2,0.1)--(8,0.1);
\draw[-,line width=0.1cm, blue] (6.4,0.1)--(8,0.1);
\node[blue] at (7.5,0.5) {$\omega_n^{(2)}=\omega_n$};
\draw[-,line width=0.1cm, blue!40] (-3.2,0.1)--(-8,0.1);
\draw[-,line width=0.1cm, blue] (-6.4,0.1)--(-8,0.1);
\node[blue] at (-7.5,0.5) {$\omega_n^{(2)}=\omega_n$};
\end{tikzpicture}
\caption{\centering For $n$ large enough the supports of $\omega_n^{(1)}$ and $\omega_n^{(2)}$, given by $B_{2M_n}$ and $\R^2\setminus B_{\frac{N_n}{2}}$, are disjoint. On $B_{M_n}$ and $\R^2\setminus B_{N_n}$ the functions $\omega_n^{(1)}$ and $\omega_n^{(2)}$ coincide  with $\omega_n$, respectively.}\label{fig:F}
\end{figure}
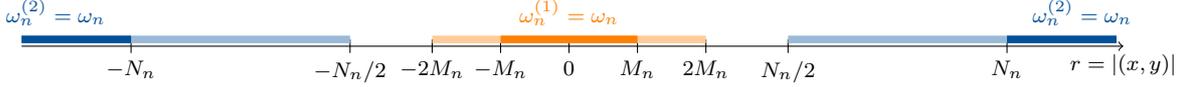
\end{center}

Roughly speaking the dichotomy assumption implies that the mass of $e_n$, which is given by $\mathcal{L}(\phi_n)=\frac{1}{2}\|\phi_n\|_{X_\frac{\alpha}{2}}^2$ splits into two disjoint regions. To be more precise, \eqref{eq:d} yields 
\begin{align}
\begin{split}\label{a_2_0_limit}
\norm{\omega_n}_{\frac{\alpha}{2}, A_{N_n,M_n}}^2&=\norm{\omega_n}_{\frac{\alpha}{2}, B_{N_n}}^2-\norm{\omega_n}_{\frac{\alpha}{2}, B_{M_n}}^2\\
&=2\left(\int_{(x_n,y_n)+B_{N_n}}e_n\ \mathrm{d}(x,y)-\int_{(x_n,y_n)+B_{M_n}}e_n\ \mathrm{d}(x,y)\right)\rightarrow 0,
\end{split}
\end{align}
as $n\rightarrow \infty$.
Using this result together with Proposition \ref{technical_results} (ii) and Proposition \ref{localization} we also find that
\begin{equation}\label{L3_limit_1}
\norm{\omega_n}_{L^3(A_{N_n,M_n})}\rightarrow 0 \qquad \mbox{for}\quad n\to \infty.
\end{equation}
 In what follows we will prove that the statements (i) and (ii) hold true.
\begin{itemize}
\item[(i)] Consider
\begin{align}
\begin{split}\label{limit_L3_dichotomy}
 \left|\mathcal{N}(\omega_n^{(1)}) + \mathcal{N}(\omega_n^{(2)}) -\mathcal{N}(\omega_n)\right|= 
\left|\int_{\mathbb{R}^2}(\omega_n^{(1)})^3\ \mathrm{d}(x,y)+\int_{\mathbb{R}^2}(\omega_n^{(2)})^3\ \mathrm{d}(x,y)-\int_{\mathbb{R}^2}\omega_n^3\ \mathrm{d}(x,y)\right|\\
\qquad \qquad =\bigg|\int_{A_{2M_n,M_n}}(\omega_n^{(1)})^3\mathrm{d}(x,y)\quad+\int_{A_{N_n,\frac{N_n}{2}}}(\omega_n^{(2)})^3\ \mathrm{d}(x,y)-\int_{A_{N_n,M_n}}\omega_n^3\ \mathrm{d}(x,y)\bigg|,
\end{split}
\end{align}
where we used that $\omega_n^{(1)}=\omega_n$ on $B_{M_n}$ and $\omega_n^{(2)}=\omega_n$ on $\R^2 \setminus B_{N_n}$.
The term $\int_{A_{N_n,M_n}}\omega_n^3\ \mathrm{d}(x,y)$ tends to zero in view of \eqref{L3_limit_1} and
\begin{align*}
\norm{w_n^{(1)}}_{L^3(A_{2M_n,M_n})}&= \norm{\partial_x \sigma_n^{(1)}}_{L^3(A_{2M_n,M_n})}\\
&\leq \frac{1}{M_n}\norm{\partial_x\chi_{1n}\sigma_{n,A_{2M_n, M_n}}}_{L^3(A_{2M_n,M_n})}+\norm{\chi_{1n}\partial_x\sigma_{n,A_{2M_n, M_n}}}_{L^3(A_{2M_n,M_n})}\\
&=\frac{1}{M_n}\norm{\partial_x\chi_{1n}\sigma_{n,A_{2M_n, M_n}}}_{L^3(A_{2M_n,M_n})}+\norm{\chi_{1n}\omega_n}_{L^3(A_{2M_n,M_n})},
\end{align*}
where we used that $\partial_x \sigma_{n,A_{2M_n, M_n}}= \partial_x \sigma_n = \omega_n$.
Using Lemma \ref{deBourd_Saut_lemma}, the smoothness of $\chi_{1,n}$, and \eqref{a_2_0_limit}, the first term on the right-hand side above can be estimated by
\begin{align}
\label{eq:est_P}
\begin{split}
\frac{1}{M_n}&\norm{\partial_x\chi_{1n}\sigma_{n,A_{2M_n, M_n}}}_{L^3(A_{2M_n,M_n})}\lesssim \frac{1}{M_n}\norm{\sigma_{n,A_{2M_n, M_n}}}_{L^3(A_{2M_n,M_n})}\\
&\lesssim M_n^{-\frac{2}{3}}\norm{\nabla\sigma_{n,A_{2M_n, M_n}}}_{L^2(A_{2M_n,M_n})}\\
&\leq M_n^{-\frac{2}{3}}\left(\norm{\partial_x\sigma_{n,A_{2M_n, M_n}}}_{L^2(A_{2M_n,M_n})} +\norm{\partial_y\sigma_{n,A_{2M_n, M_n}}}_{L^2(A_{2M_n,M_n})}\right) \\
&=M_n^{-\frac{2}{3}}\left(\norm{\omega_n}_{L^2(A_{2M_n,M_n})} +\norm{\partial_y\partial_x^{-1}\omega_n}_{L^2(A_{2M_n,M_n})}\right) \\
&\leq M_n^{-\frac{2}{3}}\|\omega_n\|_{\frac{\alpha}{2}, A_{N_n,M_n}} \to 0
\end{split}
\end{align}
as $n\to \infty$. The second term tends to zero as $n\rightarrow \infty$ due to \eqref{L3_limit_1} and the boundedness of $\chi_{1,n}$. We conclude 
\begin{equation*}
\int_{\mathbb{R}^2}(\omega_n^{(1)})^3\ \mathrm{d}(x,y)\rightarrow 0\qquad \mbox{and}\qquad \int_{\mathbb{R}^2}(\omega_n^{(2)})^3\ \mathrm{d}(x,y)\rightarrow 0
\end{equation*}
as $n\rightarrow \infty$, where the second assertion can be shown in the same way.
Together with \eqref{L3_limit_1}, equation \eqref{limit_L3_dichotomy} finishes the proof of statement (i).
\item[(ii)]


We proceed to investigate the limit 
\begin{equation*}
\lim_{n\rightarrow \infty}\mathcal{L}(\omega_n^{(1)})=\frac{1}{2}\lim_{n\rightarrow \infty}\left(\norm{\omega_n^{(1)}}_{L^2(\R^2)}^2+\norm{\mathrm{D}_x^\frac{\alpha}{2}\omega_n^{(1)}}_{L^2(\R^2)}^2+\norm{\partial_x^{-1}\partial_y\omega_n^{(1)}}_{L^2(\R^2)}^2\right)
\end{equation*}
 and show that $\lim_{n\to \infty}\mathcal{L}(\omega_n^{(1)}) = I^*$.
First consider 
\begin{align}\label{aux_omega1_L2}
\begin{split}
&\norm{\omega_n^{(1)}}_{L^2(\R^2)}^2=\norm{\frac{1}{M_n}\partial_x\chi_{1n}\sigma_{n,B_{M_n}}+\chi_{1n}\omega_n}_{L^2(\R^2)}^2\\
&\quad =\frac{1}{M_n^2}\norm{\partial_x\chi_{1n}\sigma_{n,A_{2M_n, M_n}}}_{L^2(\R^2)}^2+\frac{2}{M_n}\langle \partial_x\chi_{1n}\sigma_{n,A_{2M_n, M_n}},\chi_{1n}\omega_n\rangle_{L^2(\R^2)}+\norm{\chi_{1n}\omega_n}_{L^2(\R^2)}^2.
\end{split}
\end{align} 
Since $\partial_x \chi_{1,n}$ has support in $A_{2M_n,M_n}$ a similar argument as in \eqref{eq:est_P} shows 
\begin{equation}\label{zero_limit_L2}
\frac{1}{M_n}\norm{\partial_x\chi_{1n}\sigma_{n,A_{2M_n, M_n}}}_{L^2(\R^2)}\lesssim \norm{\omega_n}_{\frac{\alpha}{2}, A_{2M_n,M_n}}\rightarrow 0,\text{ as }n\rightarrow\infty,
\end{equation}
by using Lemma \ref{deBourd_Saut_lemma} and \eqref{a_2_0_limit}. Hence, we find that both the first and second term on the right-hand side of \eqref{aux_omega1_L2} tend to zero as $n\rightarrow \infty$. For the third term on the right-hand side of \eqref{aux_omega1_L2} we have
\begin{equation*}
\norm{\chi_{1n}\omega_n}_{L^2(\R^2)}^2=\norm{\omega_n}_{L^2(B_{M_n})}^2+\norm{\chi_{1n}\omega_n}_{L^2(A_{2M_n,M_n})}^2,
\end{equation*}
where we used that $\mbox{supp} (\chi_{1,n})\subset B_{2M_n}$ and $\chi_{1,n}=1$ on $B_{M_n}$.
Due to \eqref{a_2_0_limit} we find 
\begin{equation*}
\norm{\chi_{1n}\omega_n}_{L^2(A_{2M_n,M_n)}}\lesssim \norm{\omega_n}_{\frac{\alpha}{2},A_{2M_n,M_n}}\rightarrow 0.
\end{equation*}
We conclude 
\begin{equation}\label{L2_main_limit}
\lim_{n\rightarrow\infty}\left|\norm{\omega_n^{(1)}}_{L^2(\R^2)}^2-\norm{\omega_n}_{L^2(B_{M_n})}^2\right|=0.
\end{equation}
In the same way we can show
\begin{equation}\label{dy_L2_main_limit}
\lim_{n\rightarrow \infty}\left|\norm{\partial_x^{-1}\partial_y\omega_n^{(1)}}_{L^2(\R^2)}-\norm{\partial_x^{-1}\partial_y\omega_n}_{L^2(B_{M_n})}^2\right|=0,
\end{equation}
so that we are only left to study
\begin{align}\label{aux_Dalpha}
\begin{split}
&\norm{\mathrm{D}_x^\frac{\alpha}{2}\omega_{n}^{(1)}}_{L^2(\R^2)}^2=\norm{\mathrm{D}_x^\frac{\alpha}{2}\left(\frac{1}{M_n}\partial_x\chi_{1n}\sigma_{n,A_{2M_n, M_n}}+\chi_{1n}\omega_n\right)}_{L^2(\R^2)}^2\\
&\quad =\frac{1}{M_n^2}\norm{\mathrm{D}_x^\frac{\alpha}{2}(\partial_x\chi_{1n}\sigma_{n,A_{2M_n, M_n}})}_{L^2(\R^2)}^2+\frac{2}{M_n}\langle\mathrm{D}_x^\frac{\alpha}{2}(\partial_x\chi_{1n} \sigma_{n,A_{2M_n, M_n}}),\mathrm{D}_x^\frac{\alpha}{2}(\chi_{1n}\omega_n)\rangle_{L^2(\R^2)} \\
&\qquad +\norm{\mathrm{D}_x^\frac{\alpha}{2}(\chi_{1n}\omega_n)}_{L^2(\R^2)}^2.  
\end{split}
\end{align}
We show first
\begin{equation}\label{eq:D_x}
\|\mathrm{D}_x^\frac{\alpha}{2}\left(\partial_x \chi_{1,n} \sigma_{n,A_{2M_n,M_n}}\right)\|_{L^2(\R^2)}\lesssim M_n \|\omega_n\|_{X_\frac{\alpha}{2},A_{2M_n,M_n}},
\end{equation}
which implies by \eqref{a_2_0_limit}, the smoothness of $\chi_{1,n}$ and the boundedness of $\omega_n$ in $X_\frac{\alpha}{2}$ that the first two terms on the right-hand side of \eqref{aux_Dalpha} tend to zero as $n\to \infty$.
As in the proof of Proposition \ref{localization}, an application of Leibniz' rule for fractional derivatives yields
\begin{align*}
\|\mathrm{D}_x^\frac{\alpha}{2}\left(\partial_x \chi_{1,n} \sigma_{n,A_{2M_n,M_n}}\right)\|_{L^2(\R^2)} &\lesssim \|\sigma_{n,A_{2M_n, M_n}}\|_{L^2(A_{2M_n,M_n})} + \| \mathrm{D}_x^\frac{\alpha}{2} \sigma_{n,A_{2M_n,M_n}}\|_{L^2(A_{2M_n,M_n})}\\
&\leq 2\|\sigma_{n,A_{2M_n, M_n}}\|_{L^2(A_{2M_n,M_n})} + \| \mathrm{D}_x^\frac{\alpha}{2} \omega_n\|_{L^2(A_{2M_n,M_n})},
\end{align*}
where we used interpolation and $\partial_x \sigma_{n,A_{2M_n,M_n}}= \omega_n$ in the last inequality.
Using Lemma \ref{deBourd_Saut_lemma}, the first term on the right-hand side above can by estimated by $M_n\|\omega_n\|_{X_\frac{\alpha}{2},A_{2M_n,M_n}}$ in the same spirit as in \eqref{eq:est_P}, while the second term is bounded by $\|\omega_n\|_{X_\frac{\alpha}{2},A_{2M_n,M_n}}$. Hence, \eqref{eq:D_x} holds true and the first two terms in \eqref{aux_Dalpha} tend to zero as $n\to \infty$.

It remains to investigate the third term  in \eqref{zero_limit_L2},
given by
\[
\|\mathrm{D}_x^\frac{\alpha}{2} \left(\chi_{1,n}\omega_n\right)\|_{L^2(\R^2)}^2 = \|\mathrm{D}_x^\frac{\alpha}{2} \omega_n\|_{L^2(B_{M_n})}^2 + \|\mathrm{D}_x^\frac{\alpha}{2} \left(\chi_{1,n}\omega_n\right)\|_{L^2(A_{2M_n,M_n})}^2.
\]
Again, by applying Leibniz' rule for fractional derivatives and using that $\chi_{1,n}$ is smooth, we find that
\[
\|\mathrm{D}_x^\frac{\alpha}{2} \left(\chi_{1,n}\omega_n\right)\|_{L^2(A_{2M_n,M_n})}^2 \lesssim \|\omega_n\|_{L^2(A_{2M_n,M_n})}^2 + \|\mathrm{D}_x^\frac{\alpha}{2}\omega_n\|_{L^2(A_{2M_n,M_n})}^2 \leq \|\omega\|_{X_\frac{\alpha}{2}, A_{2M_n,M_n}} \to 0,
\]
by \eqref{a_2_0_limit}.
This allows us to conclude 
\begin{equation}\label{Dalpha_main_limit}
\lim_{n\rightarrow\infty}\bigg|\norm{\mathrm{D}_x^\frac{\alpha}{2}\omega_n^{(1)}}-\norm{\mathrm{D}_x^\frac{\alpha}{2}\omega_n}_{L^2(B_{M_n})}\bigg|=0.
\end{equation}
Gathering \eqref{L2_main_limit}, \eqref{dy_L2_main_limit},  and \eqref{Dalpha_main_limit} we have shown
\begin{equation*}
\mathcal{L}(\omega_n^{(1)})\rightarrow I^* \qquad \mbox{for}\quad n\to \infty.
\end{equation*}
In the same way we can obtain 
\begin{equation*}
\mathcal{L}(\omega_n^{(2)})\rightarrow I_\mu-I^*\qquad \mbox{for}\quad n\to \infty,
\end{equation*}
which proves statement (ii).
\end{itemize}
This concludes the proof of the proposition.

\end{proof}

By Proposition \ref{prop:vanishing} and Proposition \ref{prop:dichotomy} the scenarios of "vanishing" and "dichotomy" in Theorem \ref{cc} are ruled out and the only possibility left is the concentration scenario. Hence, there exists $\{(x_n,y_n)\}_{n\in\mathbb{N}}\subset \mathbb{R}^2$ such that for each $\varepsilon>0$ there exists $r>0$ with
\begin{equation*}
\int_{B_r(x_n,y_n)}e_n\ \mathrm{d}(x,y)\geq I_\mu-\varepsilon,\text{ for all }n\in\mathbb{N}.
\end{equation*}
This implies that, for $r$ sufficiently large,
\begin{equation}\label{min_seq_1}
\int_{\mathbb{R}^2\backslash B_r(x_n,y_n)}e_n<\varepsilon.
\end{equation}
By taking a subsequence we may assume that \eqref{min_seq_1} holds for all $n\in\mathbb{N}$. This implies in particular that 
\begin{equation}\label{L2_min_seq_1}
\norm{\phi_n(\cdot-(x_n,y_n))}_{L^2(|(x,y)|>r)}<\varepsilon.
\end{equation}
Since $\{\phi_n\}_{n\in\mathbb{N}}$ is a bounded sequence in $X_\frac{\alpha}{2}$ we may assume  $\phi_n\rightharpoonup \phi\in X_\frac{\alpha}{2}$. From Proposition \ref{embedding_lemma} we know that $X_\frac{\alpha}{2}$ is compactly embedded in $L_{\text{loc}}^2(\mathbb{R}^2)$, and therefore $\phi_n\rightarrow \phi$ strongly in $L_{\text{loc}}^2(\mathbb{R}^2)$. By combining this with \eqref{L2_min_seq_1} we can use Cantors diagonal extraction process to extract a subsequence, still denoted by $\phi_n$, converging strongly in $L^2(\mathbb{R}^2)$. Proposition \ref{technical_results} (ii) implies that for $\alpha> \tfrac{4}{5}$, $\phi_n\rightarrow \phi$ in $L^3(\mathbb{R}^2)$ as well, which yields that $\mathcal{N}(\phi)=\mu$. Finally, since $\mathcal{L}(\phi)=\frac{1}{2}\norm{\phi}_\frac{\alpha}{2}$ and the norm is weakly lower semicontinuous, we find  
\begin{equation*}
\mathcal{L}(\phi)\leq \liminf\mathcal{L}(\phi_n)=I_\mu.
\end{equation*}
Hence, $\phi$ is a solution of the minimization problem \eqref{min_problem}.	\\

Next, we investigate the lump solutions numerically. For the implementation of the iteration scheme \eqref{iteration}, we consider the space interval as $[-1024,1024]\times [-1024,1024]$ and we set $c=1$ in all the experiments. To test the efficiency of the numerical scheme we first consider the KP-I case (i.e. $\alpha=2$) where the exact analytical lump solution is given in \eqref{exact-solution}.


\begin{figure}[H]
	\begin{minipage}[t]{0.45\linewidth}
		\includegraphics[width=3.1in,height=2.3in]{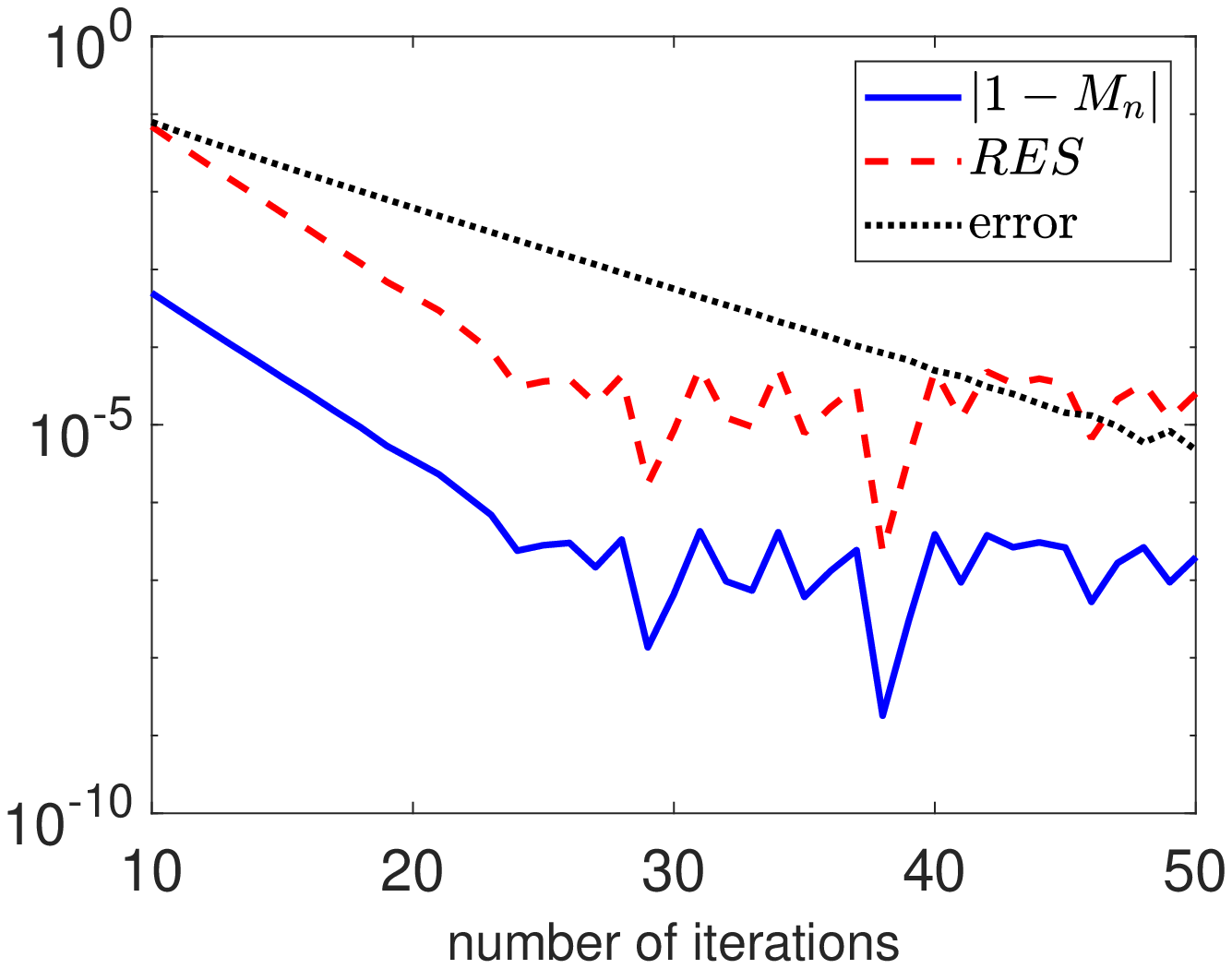}
	\end{minipage}
	\hspace{30pt}
	\begin{minipage}[t]{0.45\linewidth}
		\includegraphics[width=3.1in,height=2.3in]{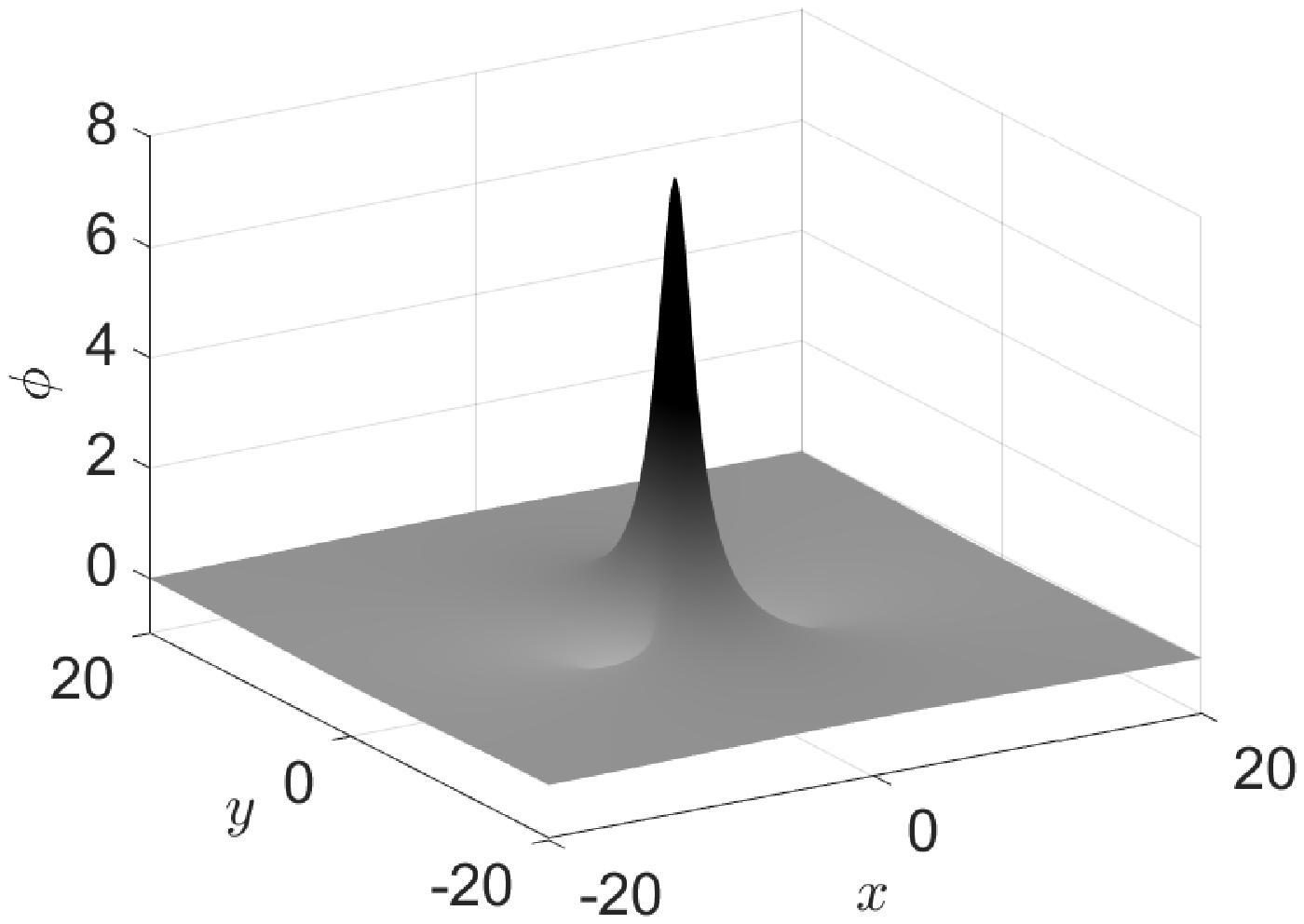}
	\end{minipage}
\hspace{4cm}
	\begin{minipage}[t]{0.45\linewidth}
		\includegraphics[width=3.1in,height=2.3in]{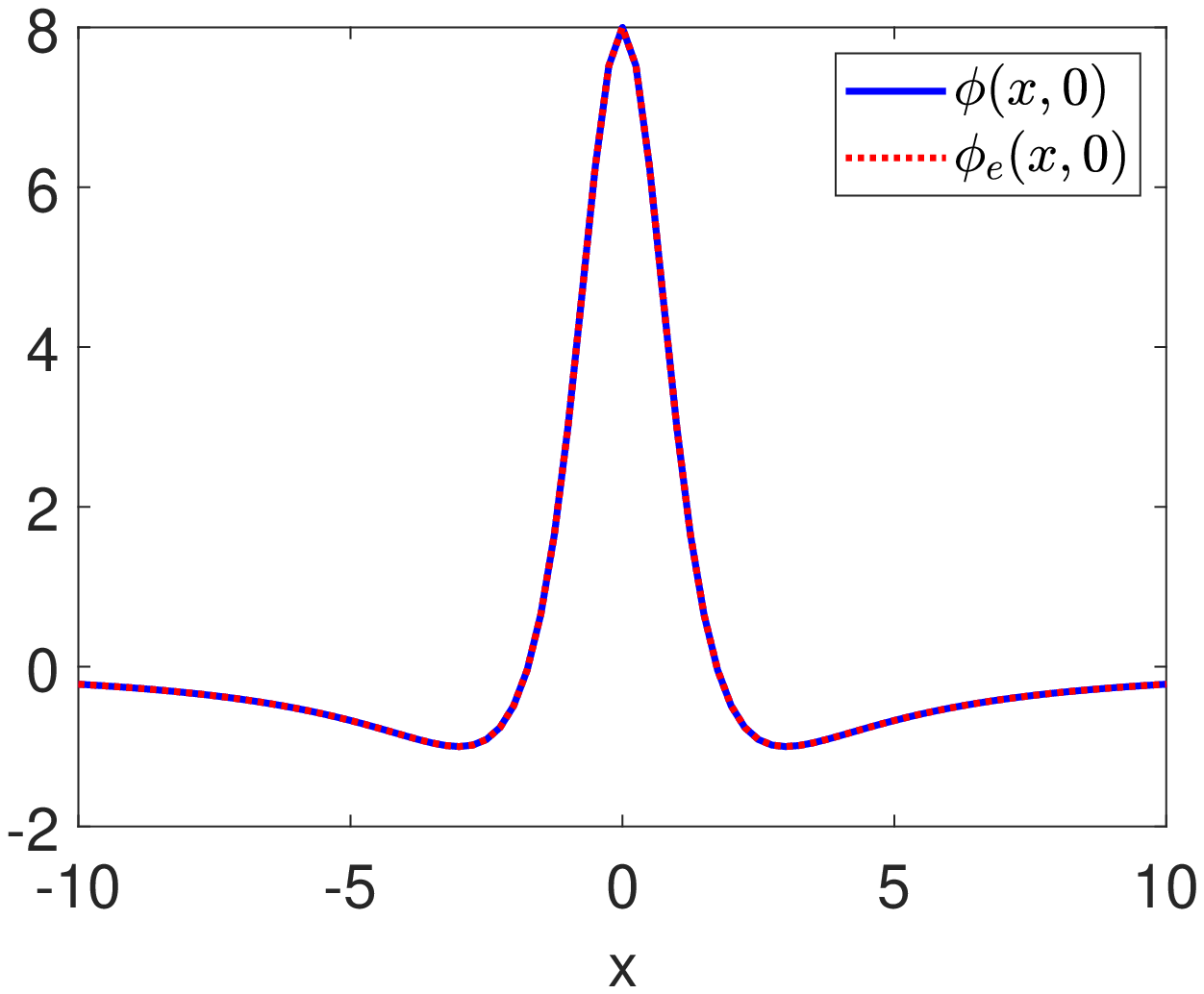}
	\end{minipage}
 \hspace{30pt}
	\begin{minipage}[t]{0.45\linewidth}
		\includegraphics[width=3.1in,height=2.3in]{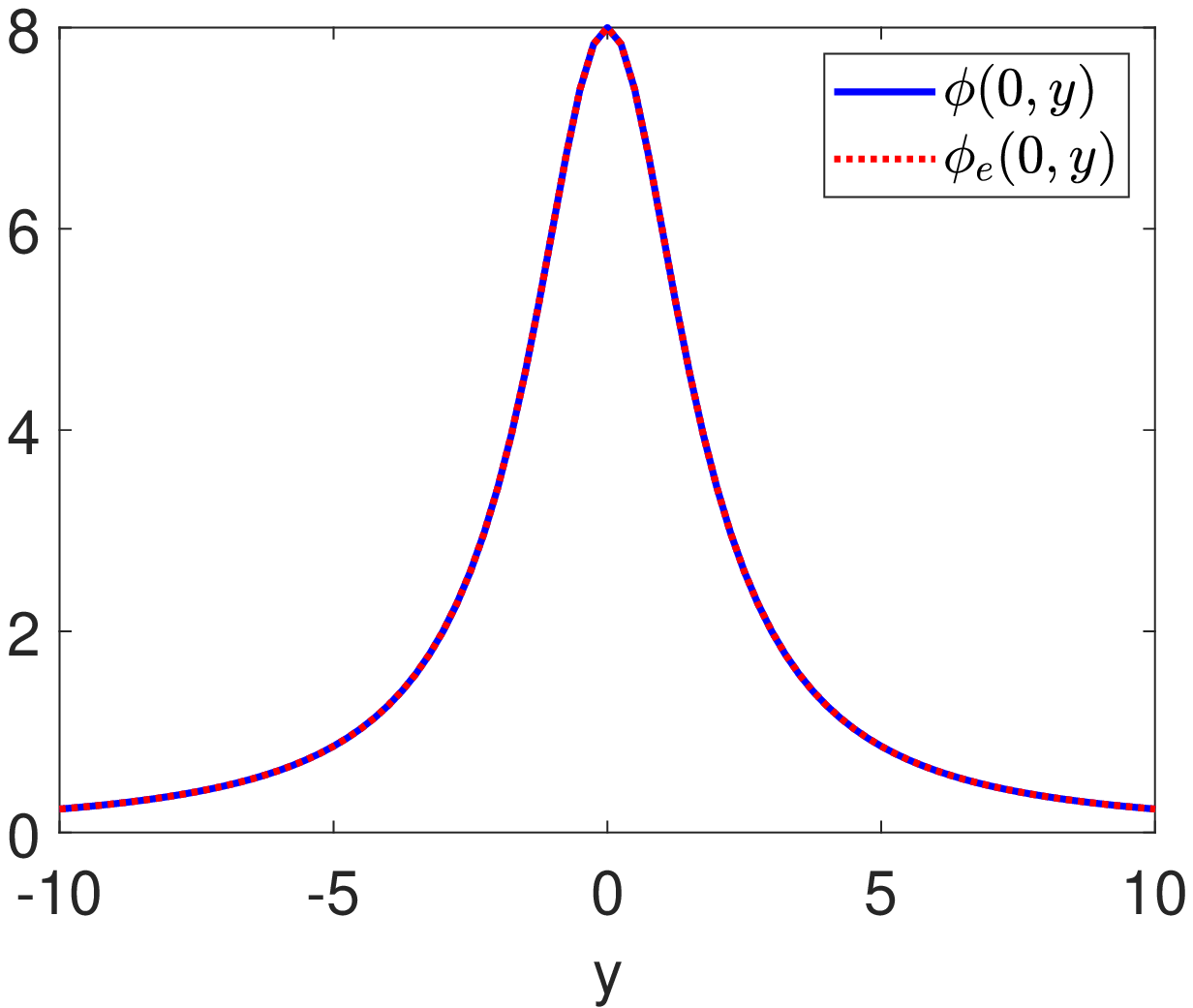}
	\end{minipage}
	\caption{The variation of the iteration, stabilization factor, and the residual errors with the number of iterations in the semi-log scale (top left), numerically generated lump solution of KP-I equation (top right), $x$-cross section $\phi(x,0)$ (bottom left) and $y$-cross section $\phi(0,y)$ (bottom right) of both numerical and analytical solutions.}
	\label{alpha2}
\end{figure}
In Figure \ref{alpha2}, we represent the numerically generated lump solution of the KP-I equation and the cross sections $\phi(x,0)$ and $\phi(0,y)$ of both numerical and analytical solutions. Choosing  the number of grid points as $N_x=N_y=2^{13}$ for both $x$ and $y$ coordinates we see that the $L^\infty$-norm of the difference of numerical and exact solutions is approximately of order $10^{-5}$ after $50$ iterations.  In Figure \ref{alpha2} we also present the variation of
three different errors with the number of iterations in semi-log scale.
\begin{figure}[H]
	\begin{minipage}[t]{0.45\linewidth}
		\includegraphics[width=3.1in,height=2.3in]{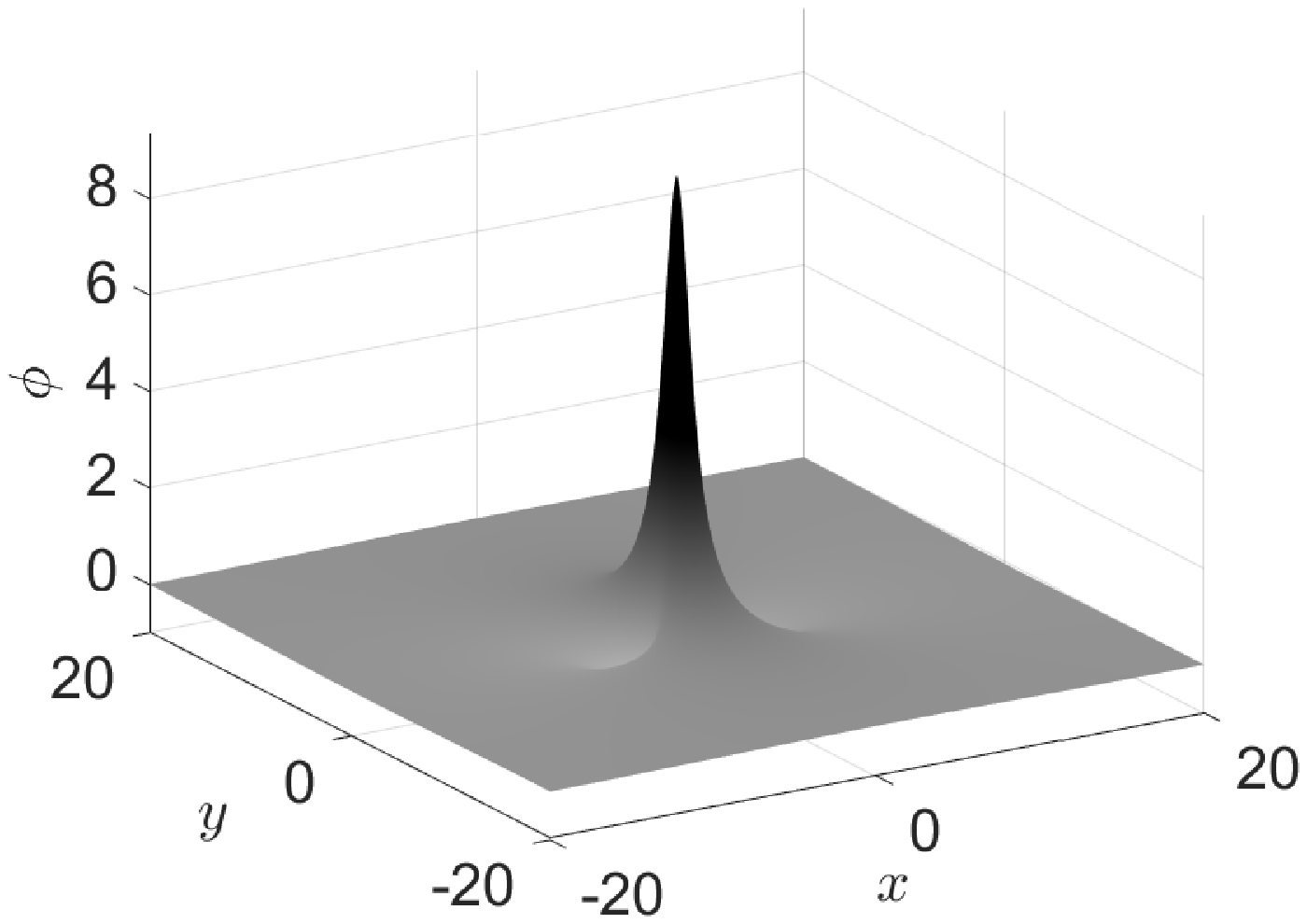}
	\end{minipage}
	\hspace{30pt}
	\begin{minipage}[t]{0.45\linewidth}
		\includegraphics[width=3.1in,height=2.3in]{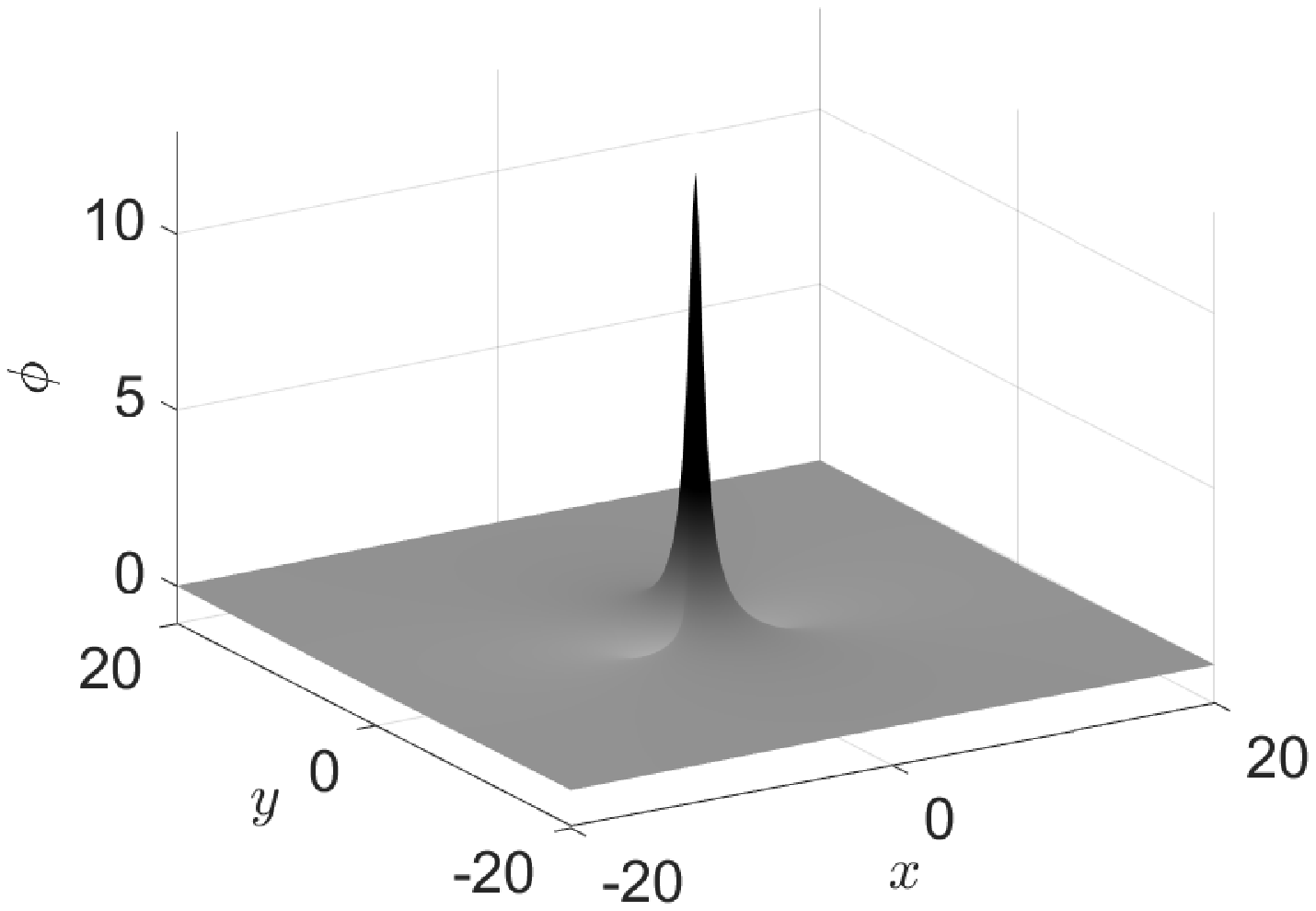}
	\end{minipage}
	\caption{ Lump solutions for $\alpha=1.7$ (left panel) and  $\alpha=1.35$ (right panel).}
	\label{alpha17-135}
\end{figure}

In the next experiment we consider the fractional case.  Figure \ref{alpha17-135}  depicts the  profiles of the numerical solutions for $\alpha=1.7$ and $\alpha=1.35$, respectively.

It can be seen from the numerical results that the lump solutions become more peaked for smaller values of $\alpha$. Therefore, to ensure  the required numerical  accuracy  we need to increase the number of grid points to $2^{14}$ for both $x$ and $y$ directions when $\alpha=1.35$.  In this case the Fourier coefficients go down to $10^{-5}$.  To obtain the same numerical accuracy for smaller values of $\alpha$, we need to increase the number of Fourier modes more, which is not accessible due to the limits of computation. \\

We also observe the cross-sectional symmetry of the lump solutions of the fKP-I equation numerically.  We present several $x$ and $y$-cross sections of the solutions for various $\alpha$. We consider the cases $\alpha=2$, $\alpha=1.7$ and $\alpha=1.35$ in Figure \ref{alpha2-symmetry}, Figure \ref{alpha17-symmetry}, and Figure \ref{alpha135-symmetry}, respectively. The numerical results indicate symmetry in both $x$ and $y$ directions. 

\begin{figure}[H]
	\begin{minipage}[t]{0.45\linewidth}
		\includegraphics[width=3.1in,height=2.3in]{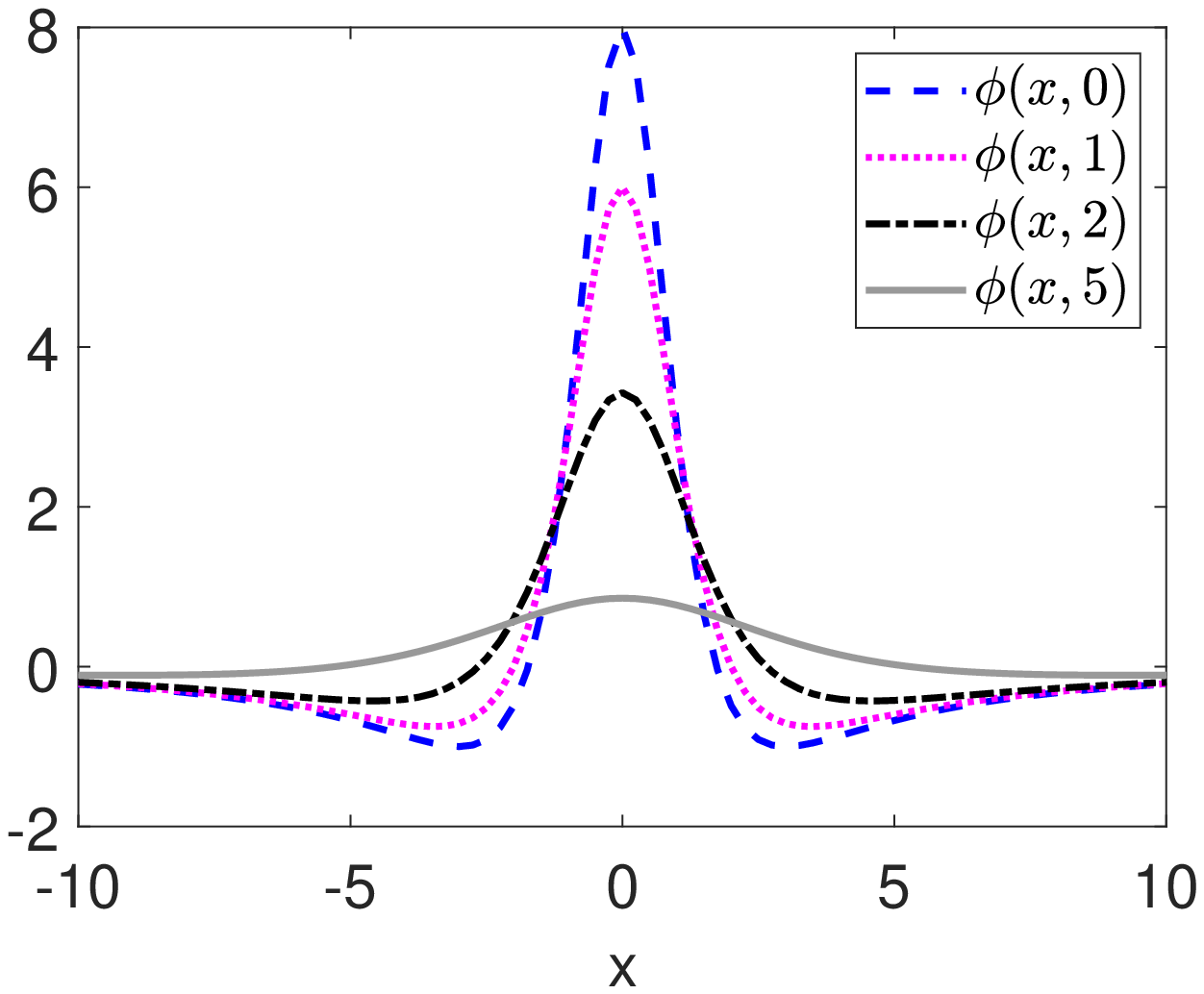}
	\end{minipage}
	\hspace{30pt}
	\begin{minipage}[t]{0.45\linewidth}
		\includegraphics[width=3.1in,height=2.3in]{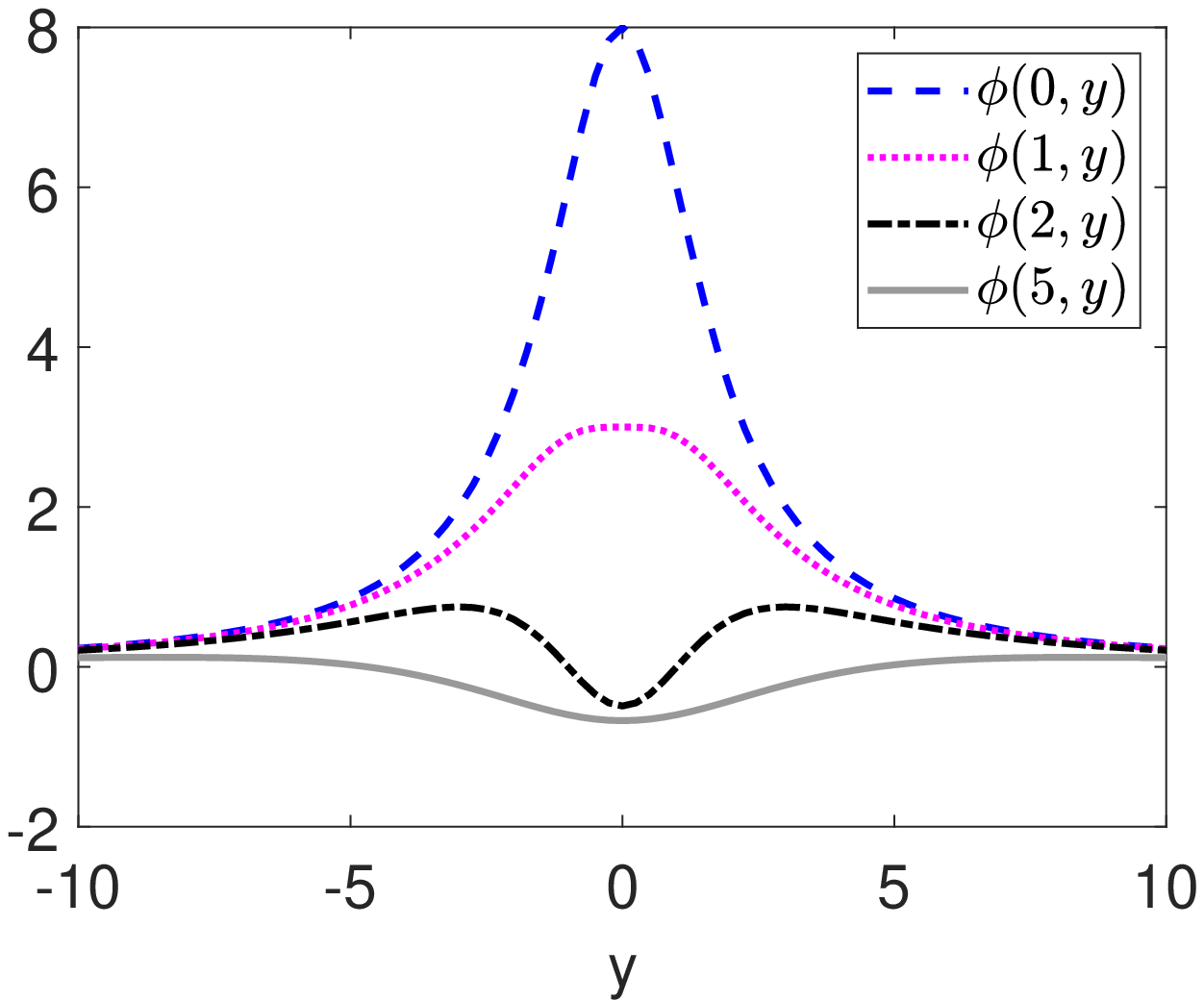}
	\end{minipage}
	\caption{Several $x$ and $y-$cross sections of the numerically generated lump solution  for $\alpha=2$.  }
	\label{alpha2-symmetry}
\end{figure}

\begin{figure}[H]
	\begin{minipage}[t]{0.45\linewidth}
		\includegraphics[width=3.1in,height=2.3in]{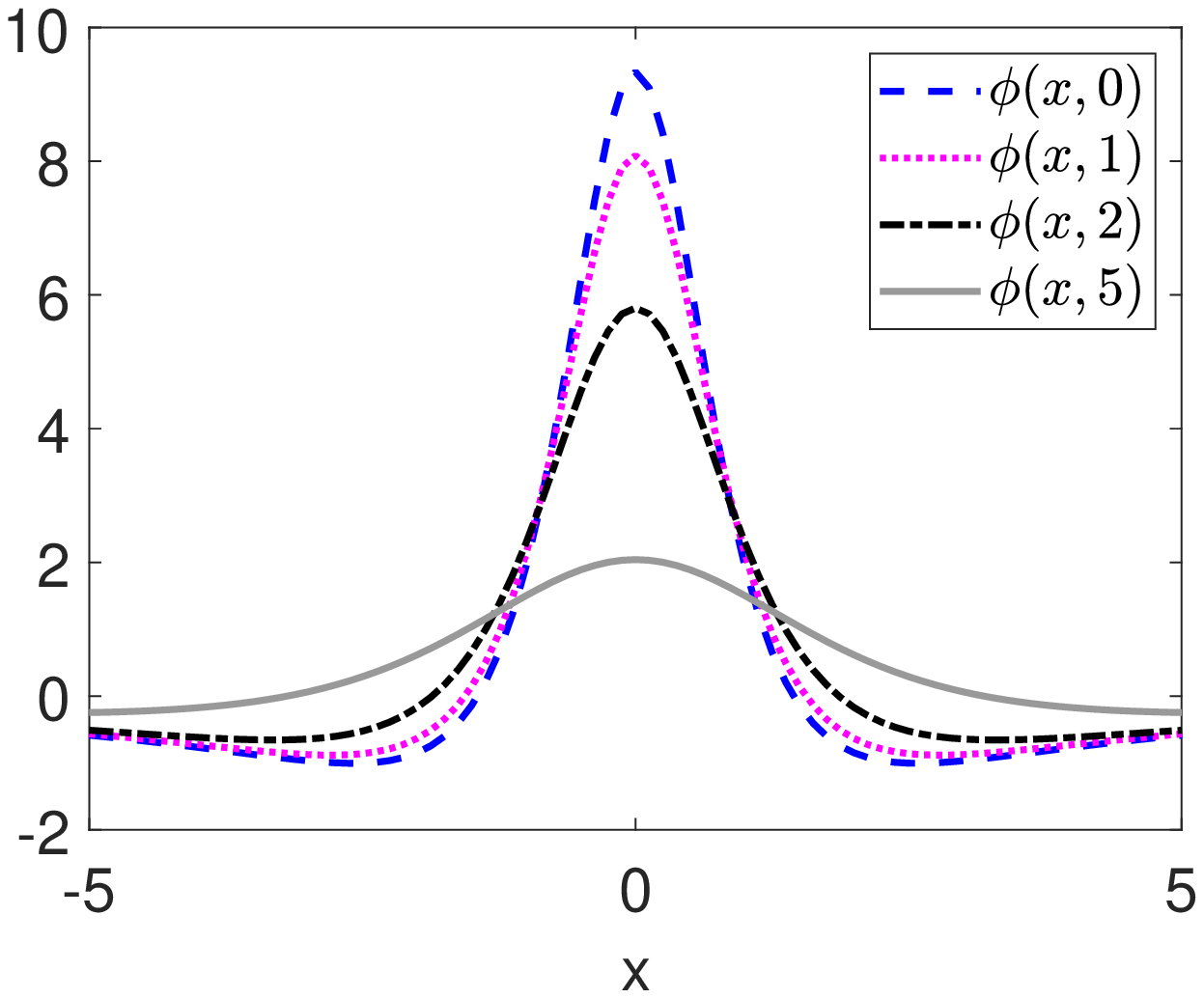}
	\end{minipage}
	\hspace{30pt}
	\begin{minipage}[t]{0.45\linewidth}
		\includegraphics[width=3.1in,height=2.3in]{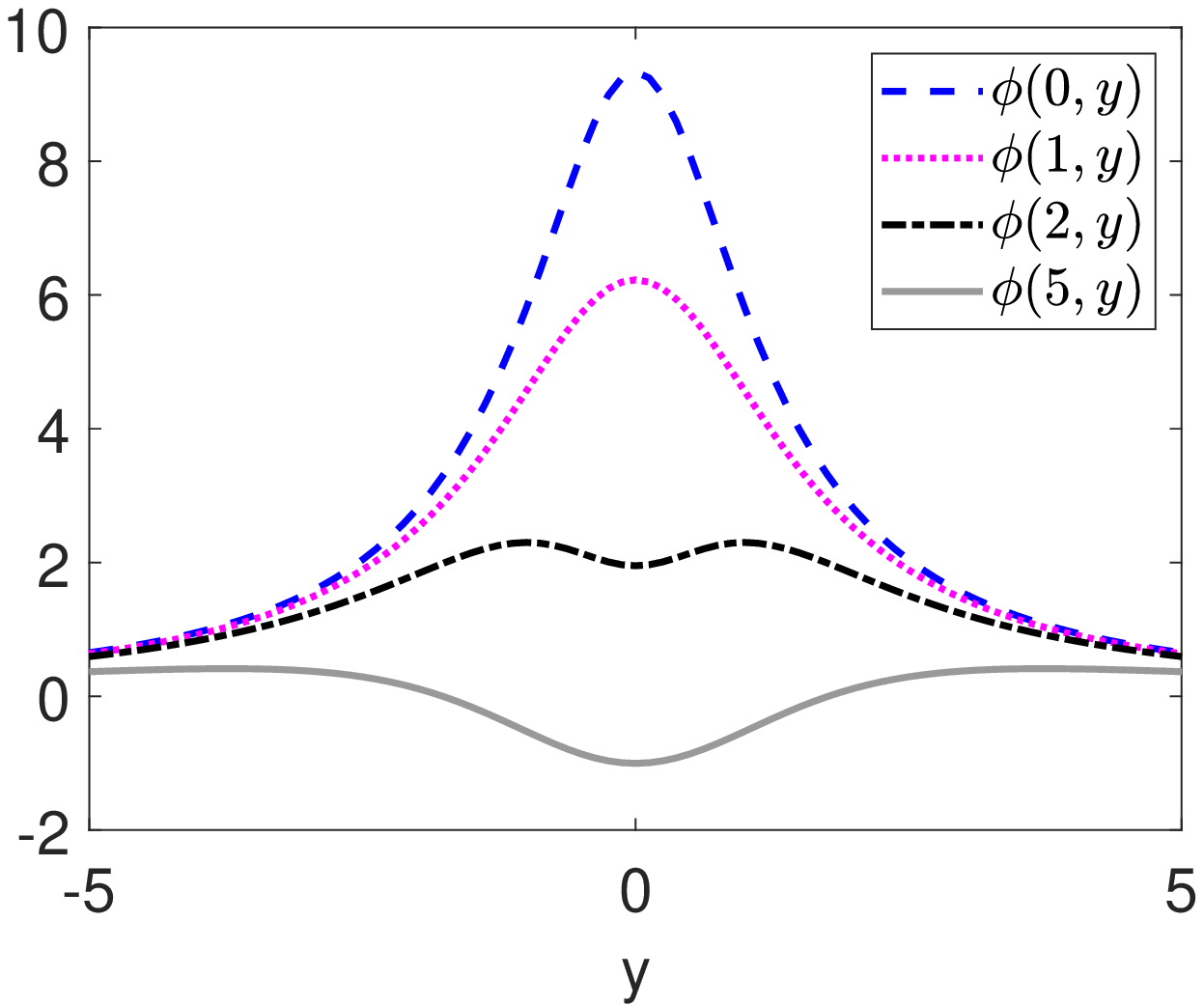}
	\end{minipage}
	\caption{Several $x$ and $y-$cross sections of the numerically generated lump solution  for $\alpha=1.7$.  }
	\label{alpha17-symmetry}
\end{figure}

\begin{figure}[H]
	\begin{minipage}[t]{0.45\linewidth}
		\includegraphics[width=3.1in,height=2.3in]{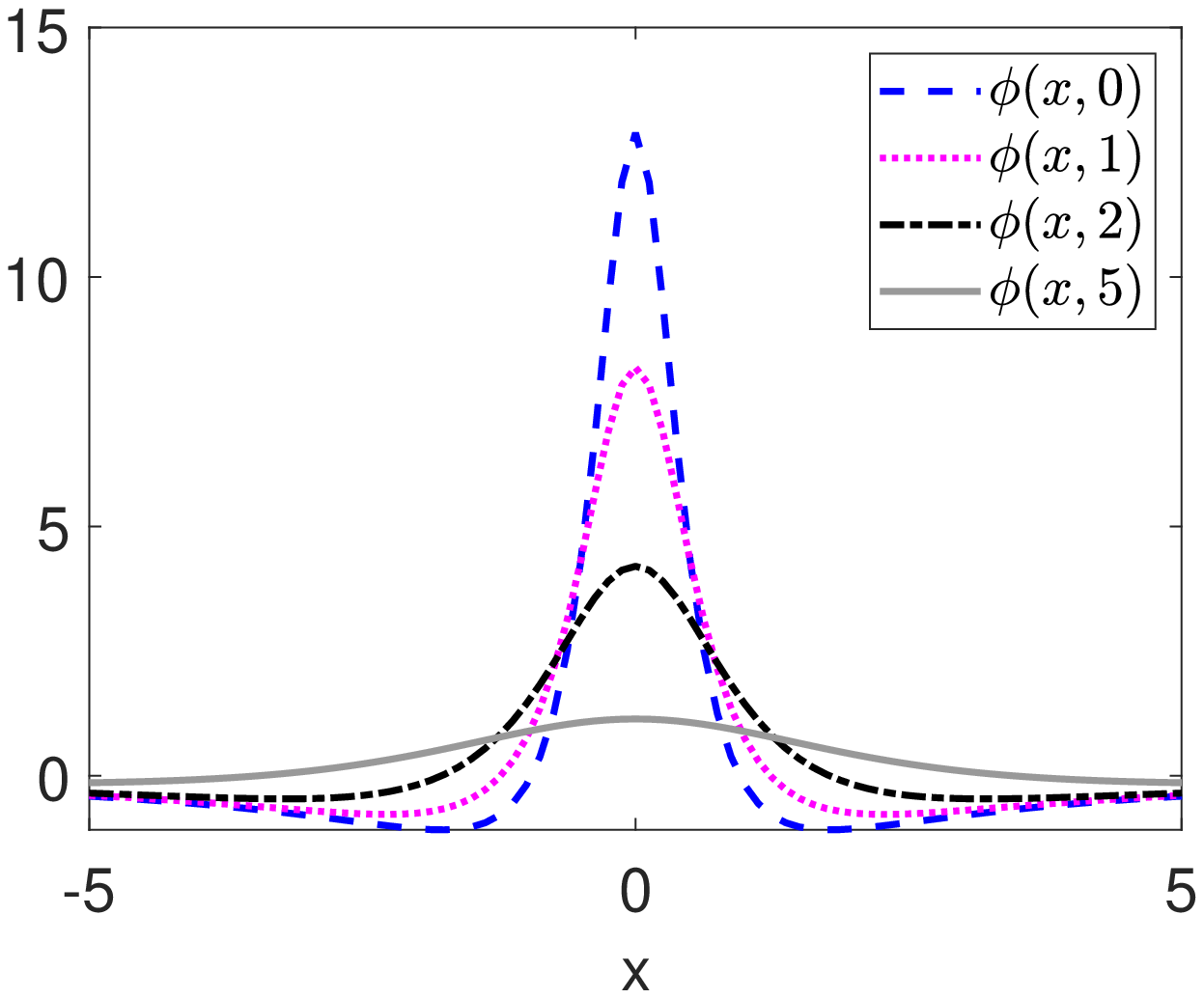}
	\end{minipage}
	\hspace{30pt}
	\begin{minipage}[t]{0.45\linewidth}
		\includegraphics[width=3.1in,height=2.3in]{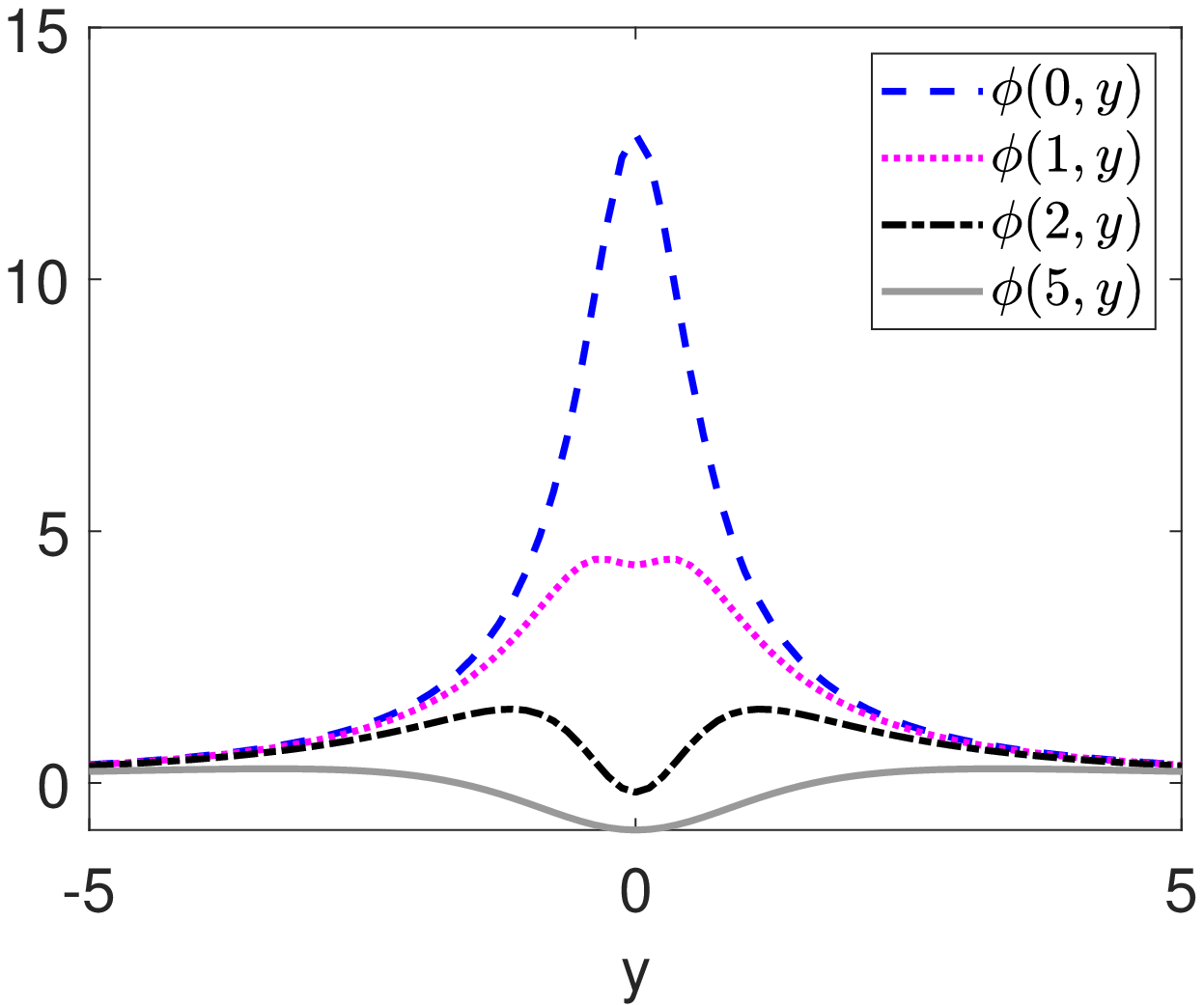}
	\end{minipage}
	\caption{Several $x$ and $y-$cross sections of the numerically generated lump solution  for $\alpha=1.35$.  }
	\label{alpha135-symmetry}
\end{figure}

\bigskip
	\section{Decay of lump solutions}\label{decay_sec}

Throughout this section, unless specifically stated otherwise, we assume that $\alpha>\frac{4}{5}$. The existence of lump solutions $u(t,x,y)=\phi(x-ct,y)$ for the fKP-I equation, where $\phi \in X_{\frac{\alpha}{2}}$, was proved in the previous section. The function $\phi$ satisfies the (rescaled) traveling wave fKP-I equation 
\begin{equation}\label{eq:eq}
-\phi_{xx}-\phi_{yy}- \mathrm{D}_x^\alpha \phi_{xx}+\frac{1}{2} (\phi^2)_{xx}=0,
\end{equation}
which can be written in convolution form as	\begin{equation}\label{eq:convolution}
		\phi=\frac{1}{2}K_\alpha*\phi^2,\qquad \hat K_\alpha(\xi_1,\xi_2)= m_\alpha(\xi_1,\xi_2),
	\end{equation}
where the symbol $m_\alpha$ is given by
\[
 m_\alpha(\xi_1,\xi_2)=\frac{\xi_1^2}{|\xi|^2+|\xi_1|^{\alpha +2}}.
\]
Let us recall from Remark \ref{rem:optimal} that any nontrivial, continuous solution $\phi$ of \eqref{eq:convolution} decays \emph{at most} quadratically. In this section, we show that any nontrivial solution $\phi\in X_{\frac{\alpha}{2}}$ of \eqref{eq:convolution} decays indeed \emph{exactly} quadratically, that is we prove Theorem \ref{thm:decay}.

\medskip


The idea is to study the kernel function $K_\alpha$ and to show that it has quadratic decay at infinity (independent of $\alpha$). Then the decay properties of $K_\alpha$ are used to show that also $\phi$ decays quadratically at infinity. In the sequel we denote by $r:\R^2 \to \R$ the function \[r(x,y)=|(x,y)|=(x^2+y^2)^{\frac{1}{2}}.\] It will be useful to note that for all $a\geq 1$ we have that $r^a$ is convex, so that 
\begin{equation}\label{eq:convex}
r^a(x,y)\lesssim r^a(x-\bar x, y-\bar y)+r^a(\bar x, \bar y)\qquad \mbox{for all}\quad (x,y),(\bar x, \bar y) \in \R^2.
\end{equation}


Notice that by \eqref{eq:convex} and Young's inequality
\begin{align*}
\|r^2 \phi\|_\infty &\lesssim \|r^2K_\alpha*\phi^2\|_\infty + \|K_\alpha * r^2\phi^2\|_\infty\\&
\lesssim \|r^2K_\alpha\|_\infty \|\phi\|_{L^2(\R^2)}^2 + \|K_\alpha\|_{L^q(\R^2)}\|r^2\phi^2\|_{L^{q^\prime}(\R^2)}
\end{align*}
for some $1\leq q,q^\prime\leq \infty$ with $1=\frac{1}{q}+\frac{1}{q^\prime}$,
so that the statement is proved provided that 
\begin{itemize}
\item[(A)] $r^2K \in L^\infty(\R^2)$
\item[(B)] there exists $1\leq q\leq \infty$ such that $K_\alpha \in L^q(\R^2)$ and $r^2\phi^2 \in L^{q^\prime}(\R^2)$, where $q^\prime$ is the dual conjugate of $q$.
\end{itemize}

Before studying the properties of the kernel function $K_\alpha$, we state the following  two lemmata, which yield some a priori regularity of lump solutions in the energy space.

\begin{lemma}\label{lem:smooth}
Any solution $\phi$ of \eqref{eq:convolution} in the energy space $X_{\frac{\alpha}{2}}$  satisfies $\phi \in L^r(\R^2)$ for all $2\leq r<\infty$ and
\[
\phi \in H^\infty(\R^2).
\]
In particular, $\phi$ is uniformly continuous and decays to zero at infinity.
\end{lemma}

\begin{proof}
Let us start by repeating the Hörmander--Mikhilin multiplier theorem \cite{Lizorkin}, which states that if  $f:\R^2 \to \R$ is a function, which is smooth outside the origin and 
\[
\xi \mapsto \xi_1^{k_1}\xi_2^{k_2} \frac{\mathrm{d}^k}{\mathrm{d} \xi_1^{k_1}\mathrm{d} \xi_2^{k_2}}f(\xi)
\]
 is bounded on $\R^2$ for all $k_1,k_2 \in \{0,1\}$ with $k=k_1+k_2 \in\{0,1,2\}$, then $f$ is a Fourier multiplier on $L^p(\R^2)$ for all $1<p<\infty$, i.e. the operator $T_f$ defined by $T_f g= \mathcal F^{-1}\left(f \hat g\right)= \mathcal{F}^{-1}(f)*g$ is bounded on $L^p(\R^2)$.
By the Hörmander--Mikhilin multiplier theorem, it is easy to check that the functions
\[
\xi \mapsto m_\alpha(\xi),\qquad \xi \mapsto |\xi_1|^\alpha m_\alpha(\xi),\qquad \xi \mapsto \xi_2m_\alpha(\xi)
\]
are Fourier multipliers on $L^p(\R^2)$ for $1<p<\infty$.
Let $\phi \in  X_{\frac{\alpha}{2}}$. Due to Proposition \ref{technical_results} (ii), we have that $\phi \in L^3(\R^2)$, which implies  $\phi^2 \in L^{\frac{3}{2}}(\R^2)$.
Since
\begin{align*}
\phi = \frac{1}{2}\mathcal{F}^{-1}(m_\alpha)*\phi^2, \qquad 
\mathrm{D}^\alpha_x \phi = \frac{1}{2}\mathcal{F}^{-1}(|\xi_1|^\alpha m_\alpha)*\phi^2,\qquad
\phi_y = -\frac{\ii}{2}\mathcal{F}^{-1}(\xi_2m_\alpha)*\phi^2,
\end{align*}
we find 
\[
\phi, \mathrm{D}^\alpha_x \phi, \phi_y \in L^{\frac{3}{2}}(\R^2).
\]
In particular, $\phi$ belongs to the anisotropic Sobolev space $W^{\vec{\alpha}, \frac{3}{2}}(\R^2)$ for $\vec \alpha = (\alpha,1)$, where
\[
W^{\vec{\alpha}, q}(\R^2):=\{\phi \in L^q(\R^2)\mid \mathrm{D}_x^{\alpha_1} \phi, \mathrm{D}_y^{\alpha_2}\phi \in L^q(\R^2)\}.
\]
We use the following anisotropic Gagliardo--Nirenberg inequality for fractional derivatives \cite[Theorem 1.1]{Esfahani}: If $\phi \in W^{\vec{\alpha}, q}(\R^2)$ with  $A:=\frac{1}{q}\left(\frac{1}{\alpha_1}+\frac{1}{\alpha_2}\right)-1>0$ and $M:=1+(\frac{1}{p}-\frac{1}{q})\left(\frac{1}{\alpha_1}+\frac{1}{\alpha_2}\right)>0$, then 
\begin{equation}\label{eq:fGN}
\|\phi\|_{L^r(\R^2)}\lesssim \|\phi\|_{L^p(\R^2)}^{1-\theta}\|\mathrm{D}_x^{\alpha_1}\phi\|_{L^q(\R^2)}^{\theta_1}\|\mathrm{D}_y^{\alpha_2}\phi\|_{L^q(\R^2)}^{\theta_2},
\end{equation}
for all 
\[
p\leq r<r_*:=\frac{1}{A }\left(\frac{1}{\alpha_1}+\frac{1}{\alpha_2}\right),
\]
where $\theta = \theta_1 + \theta_2$ and $\theta_i= (\frac{1}{p}-\frac{1}{r})(\alpha_i M)^{-1}$. Applied to the situation at hand, we can choose $p=2$, $q=\frac{3}{2}$ and $\vec \alpha=(\alpha,1)$ for $\alpha = \frac{4}{5+4\e}$ with $\e>0$ arbitrarily small, which yields $A= \frac{1}{2}+\frac{2}{3}\e$ and $M=\frac{5}{8}-\frac{1}{6}\e$. Due to \eqref{eq:fGN} we find that
\[
\phi \in L^r(\R^2)\qquad \mbox{for all} \quad 2\leq r < \frac{9}{2}.
\]
Repeating the same argument for $\phi \in L^{\frac{9}{2}-2\e}(\R^2)$ with $\phi^2 \in L^{\frac{9-4\e}{4}}(\R^2)$, we find that $\phi \in W^{(\alpha,1), \frac{9-4\e}{4}}(\R^2)$ and again by the fractional Gagliardo--Nirenberg inequality for $p=2$, $q=\frac{9-4\e}{4}$, $\vec \alpha=(\alpha,1)$ for  $\alpha = \frac{4}{5+4\e}$, we obtain that $A=\frac{8\e}{9-4\e}$ and $M=\frac{9}{8}+\frac{\e}{2}\frac{4\e+7}{4\e-9}$, so that
\begin{equation}\label{eq:phi_r}
\phi \in L^r(\R^2)\qquad \mbox{for all} \quad 2\leq r < \infty,
\end{equation}
by letting $\e \to 0$. This proves the first assertion. The relation \eqref{eq:phi_r} implies by the Fourier multiplier theorem that
\[
\phi, \mathrm{D}_x^\alpha \phi, \phi_y \in L^r(\R^2)\qquad \mbox{for all} \quad 2\leq r < \infty.\]
Next, we aim to bootstrap the smoothness. By Hölder's inequality it is clear also that  $(\phi^2)_y \in L^r(\R^2)$ for all $2\leq r<\infty$. Due to the Leibniz rule for fractional derivatives (see e.g. \cite[Theorem 7.6.1]{Grafakos}) and Hölder's inequality we can estimate
\begin{align*}
\|\mathrm{D}_x^\alpha \phi^2\|_{L^r(\R^2)}^r &= \int_{\R} \|\mathrm{D}_x^\alpha\phi (\cdot, y)\|_{L^r(\R)}^r\, \mathrm{d}y\\
&\lesssim \int_{\R}\|\mathrm{D}_x^\alpha\phi (\cdot, y)\|_{L^{2r}(\R)}^r\|\phi (\cdot, y)\|_{L^{2r}(\R)}^r\, \mathrm{d}y\\
&\leq \left(\int_{\R}\|\mathrm{D}_x^\alpha\phi (\cdot, y)\|_{L^{2r}(\R)}\,\mathrm{d}y\right)^\frac{1}{2}\left(\int_{\R}\|\phi (\cdot, y)\|_{L^{2r}(\R)}^{2r}\,\mathrm{d}y\right)^\frac{1}{2}\\
&= \|\mathrm{D}_x^\alpha \phi\|_{L^{2r}(\R^2)}^\frac{r}{2}\|\phi\|_{L^{2r}(\R^2)}^\frac{r}{2},
\end{align*}
which yields that also $\mathrm{D}_x^\alpha \phi^2 \in L^r(\R^2)$ for all $2\leq r <\infty$. This can be used to bootstrap the smoothness of $\phi$, by using
\[
\mathrm{D}^\alpha_x \phi = \frac{1}{2}K_\alpha* \mathrm{D}^\alpha_x \phi^2\qquad \mbox{and} \qquad \phi_y = \frac{1}{2}K_\alpha*(\phi^2)_y.
\]
Reiterating the argument yields
\[
\phi, \mathrm{D}_x^\alpha\phi, \phi_y, \mathrm{D}_x^{2\alpha} \phi,\mathrm{D}_x^{\alpha}\phi_y, \phi_{yy} \in L^r(\R^2)\qquad \mbox{for all} \quad 2 \leq r < \infty.
\]
and eventually $\mathrm{D}_x^k \phi \in L^r(\R^2)$ for all $2\leq r <\infty$ and $k\in \N$, which implies that $\phi \in H^\infty(\R^2)$. Eventually, since $H^\infty(\R^2)$ is embedded into the space of uniformly continuous functions on $\R^2$ and $\phi$ is $L^2(\R^2)$-integrable, we deduce that $\phi$ decays to zero at infinity.
\end{proof}
 
	\begin{lemma}\label{lem:est}
		Any solution $\phi$ of \eqref{eq:eq} in the energy space $X_{\frac{\alpha}{2}}$  satisfies
		\[
		\int_{\R^2} (x^2+y^2)(|\phi_x|^2+|\phi_y|^2+|\mathrm{D}^\frac{\alpha}{2}_x\phi_x|^2)\,\mathrm{d}(x,y)<\infty.
		\]
	\end{lemma}

 \begin{proof}
 The proof follows essentially the lines in \cite[Lemma  3.1]{deBouard}. Here, we proceed formally by omitting the truncation function at infinity. First, let us multiply \eqref{eq:eq} by $x^2\phi$ and integrate over $\R^2$. Then
 \[
0= \int_{\R^2} x^2 \phi \left(-\phi_{xx}-\phi_{yy} -\mathrm{D}^\alpha_x \phi_{xx}+\frac{1}{2}\left(\phi^2\right)_{xx}\right)\,\mathrm{d}(x,y).
 \]
 Using integration by parts we find that
 \begin{align*}
 \int_{\R^2} x^2 \phi \left(-\phi_{xx}-\phi_{yy}+\frac{1}{2}\left(\phi^2\right)_{xx}\right)\,\mathrm{d}(x,y)=\int_{\R^2} x^2 \left(\phi_x^2+ \phi_y^2-\phi \phi_x^2\right) -\phi^2 +\frac{2}{3}\phi^3 \,\mathrm{d}(x,y).
 \end{align*}
In view of Lemma \ref{lem:frac_int}, the nonlocal part can be written as
 \begin{align*}
 -\int_{\R^2}x^2\phi \mathrm{D}^\alpha_x \phi_{xx}\, d(x,y)&=-\int_{\R^2} x^2\phi_{xx}\mathrm{D}^\alpha_x \phi\,\mathrm{d}(x,y)-\int_{\R^2} 4x\phi_x \mathrm{D}^\alpha_x \phi\,d(x,y)-\int_{\R^2}  2\phi \mathrm{D}^\alpha_x \phi\,\mathrm{d}(x,y)\\
 &=  \int_{\R^2}x^2\left(\mathrm{D}^{\frac{\alpha}{2}}\phi_x\right)^2\,\mathrm{d}(x,y) - \frac{1}{4}(\alpha+2)^2\int_{\R^2}\left(\mathrm{D}^{\frac{\alpha}{2}}\phi\right)^2\,\mathrm{d}(x,y).
 \end{align*}
 Adding the above equalities we obtain that
 \begin{align}\label{eq:x}
 \int_{\R^2} x^2 \left(\phi_x^2+ \phi_y^2+\left(\mathrm{D}^{\frac{\alpha}{2}}\phi_x\right)^2\right) \,\mathrm{d}(x,y)=\int_{\R^2}x^2\phi\phi_x^2 + \phi^2 -\frac{2}{3}\phi^3+\frac{1}{4}(\alpha+2)^2\left(\mathrm{D}^{\frac{\alpha}{2}}\phi\right)^2\,\mathrm{d}(x,y)
 \end{align}
 Multiplying \eqref{eq:eq} by $y^2\phi$ instead yields
 \[
\int_{\R^2} y^2 \phi \left(-\phi_{xx}-\phi_{yy} -\mathrm{D}^\alpha_x \phi_{xx}+\frac{1}{2}\left(\phi^2\right)_{xx}\right)\,\mathrm{d}(x,y)=0.
 \]
 Again, using integration by parts, we find that
 \begin{align}\label{eq:y}
\int_{\R^2}y^2 \left(\phi_x^2+\phi_y^2+\left(\mathrm{D}^\frac{\alpha}{2}_x\phi_x\right)^2\right)\, \mathrm{d}(x,y) = \int_{\R^2} \phi^2+y^2\phi \phi_x^2\,\mathrm{d}(x,y).
 \end{align}
 Adding \eqref{eq:x} and \eqref{eq:y}, while keeping in mind that $\phi \in X_{\frac{\alpha}{2}}$, we can estimate
\[
\int_{\R^2} (x^2+y^2)(|\phi_x|^2+|\phi_y|^2+|\mathrm{D}^\frac{\alpha}{2}_x\phi_x|^2)\, \mathrm{d}(x,y)\lesssim 1 + \int_{\R^2} (x^2+y^2)\phi\phi_x^2\,\mathrm{d}(x,y).
\]
Using that $\phi$ is continuous and tends to zero at infinity (see Lemma \ref{lem:smooth}), there exists $R>0$ such that $\phi(x,y)\leq \frac{1}{2}$ for $|(x,y)|\geq R$ and we conclude 
\[
\int_{\R^2} (x^2+y^2)(|\phi_x|^2+|\phi_y|^2+|\mathrm{D}^\frac{\alpha}{2}_x\phi_x|^2)\lesssim_R 1.
\]
  
\end{proof}

\subsubsection*{Properties of the kernel function $K_\alpha$.}
	We will first concentrate on the regularity properties of $K_\alpha$.
\begin{lemma}\label{lem:m}
	$m_\alpha\in L^p(\R^2)$ if and only if $p>\frac{2}{\alpha}+\frac{1}{2}$.
\end{lemma}	

\begin{proof}
It is clear that $m_\alpha \in L^\infty(\R^2)$. Let us  compute
\begin{align*}
\|m_\alpha\|_{L^p(\R^2)}^p &= \int_{\R^2} \frac{1}{\left(1+|\xi_1|^\alpha + \frac{\xi_2^2}{\xi_1^2}\right)^p}\, \mathrm{d}(x,y)\\
&= \int_{\R} \frac{1}{\left(1+|\xi_1|^\alpha\right)^p} \int_{\R} \frac{1}{\left(1+\frac{\xi_2^2}{(1+|\xi_1|^\alpha)\xi_1^2}\right)^p} \, \mathrm{d}\xi_2\,\mathrm{d}\xi_1 \\
&=\int_{\R} \frac{|\xi_1|}{\left(1+|\xi_1|^\alpha\right)^{p-\frac{1}{2}}} \,\mathrm{d}\xi_1 \int_{\R} \frac{1}{\left(1+z^2\right)^p}\,\mathrm{d}z,
\end{align*}
where we used the change of variables $z=\tfrac{\xi_2}{|\xi_1|(1+|\xi_1|^\alpha)^\frac{1}{2}}$.
Since the second integral above is clearly convergent for any $p\geq 1$, we find that $m_\alpha \in L^p(\R^2)$ if and only if
\[
p>\frac{2}{\alpha}+\frac{1}{2}.
\]
\end{proof}

\begin{remark}
\emph{
The above lemma implies that $m_\alpha \in L^2(\R^2)$ if and only if $\alpha>\frac{4}{3}$, which is the $L^2$-critical exponent. In this case it follows immediately, that also $K_\alpha \in L^2(\R^2)$ and the proof of Theorem \ref{thm:decay} can be done essentially by following the lines in \cite{deBouard}. In the supercritical case $\tfrac{4}{5}<\alpha\leq \frac{4}{3}$, which in particular includes the Benjamin--Ono KP equation for $\alpha=1$, the symbol $m_\alpha$ belongs to an $L^p$-space with $p>2$ so that the integrability properties of the kernel $K_\alpha$ are a priori not clear.
}
\end{remark}

\begin{lemma}\label{lem:K_smoothness}
The kernel function $K_\alpha$ is smooth outside the origin.
\end{lemma}

\begin{proof}
Let $\chi: \R^2 \to \R$ be a compactly supported, radial, smooth function with $\chi(0,0)=1$. Set $\bar m_\alpha:=\chi m_\alpha$. Then $\bar m_\alpha$ has compact support and $\mathcal{F}^{-1}( \bar m_\alpha)$ is real analytic. Now, set $\tilde m_\alpha:=(1-\chi)m_\alpha$. Then $\tilde m_\alpha$ is smooth. Let us fix $(x_0,y_0) \neq (0,0)$ and let $\psi:\R^2\to \R$ be a compactly supported, smooth function with $\psi (x,y)=1$ in an arbitrarily small neighborhood of $(x_0,y_0)$ and $\psi (0,0)=0$. Then also
\[
\Psi_k(x,y)=|(x,y)|^{-2k}\psi
\]
is smooth and compactly supported. Notice that
$
\hat \Psi_k  = -\Delta^{-k} \hat \psi
$
and
\[
\tilde m_\alpha * \hat \psi = -\tilde m_\alpha * \Delta^k\hat \Psi_k = -(\Delta^k \tilde m_\alpha * \hat \Psi_k).
\]
Since $\Psi_k$ is smooth with compact support, we know that $\hat \Psi_k \in \mathcal{S}(\R^2)$. Furthermore $\Delta^k \tilde m_\alpha$ is smooth with $\Delta^k \tilde m_\alpha(\xi) \lesssim \tfrac{1}{|\xi|^{\alpha + 2k}}$ for $|\xi| \to \infty$. Since the convolution of two integrable, smooth functions is smooth and decays at least as fast as the function with the lower decay,  we deduce that $\tilde m_\alpha *\hat \psi$ is smooth and decays at least as $\frac{1}{|\cdot|^{\alpha+2k}}$ at infinity for an arbitrary choice of $k\in \N$. In particular, $\mathcal{F}^{-1}(\tilde m_\alpha *\hat \psi)= \mathcal{F}^{-1}(\tilde m_\alpha) \psi$ is smooth, which yields that $\mathcal{F}^{-1}(\tilde m_\alpha)$ is smooth outside the origin. We conclude that
\[
K_\alpha = \mathcal{F}^{-1}( \bar m_\alpha) + \mathcal{F}^{-1}(\tilde m_\alpha)
\]
is smooth outside the origin.
\end{proof}

Let us now investigate the behavior of $K_\alpha$ at infinity. We show that the decay is quadratic, independently of the value of $\alpha>0$.

\begin{proposition}\label{prop:K_decay} For any $\alpha>0$, we have that $r^2K_\alpha$ belongs to $L^\infty(\R^2)$. 
\end{proposition}

\begin{proof}
	We have $K_\alpha=\mathcal{F}^{-1}(m_\alpha)$, so that
	\begin{align*}
	K_\alpha(x,y)&=\int_{\R^2}\frac{\xi_1^2}{\xi_1^2 + \xi_2^2 +| \xi_1|^{\alpha +2}}e^{\mathrm{i}x\xi_1 + \mathrm{i}y\xi_2}\,\mathrm{d}\xi_1\, \mathrm{d}\xi_2=\int_\R \frac{|\xi|}{(1+|\xi|^\alpha)^{\frac{1}{2}}}e^{-|y||\xi|(1+|\xi|^\alpha)^{\frac{1}{2}}}e^{\mathrm{i}x\xi}\, \mathrm{d}\xi,
	\end{align*}
 where we used that
 \[
\mathcal{F}\left(\frac{1}{a^2+(\cdot)^2}\right)(y)=\int_{\R} \frac{1}{a^2+\xi_2^2}e^{-\mathrm{i}\xi_2 y}\, \mathrm{d}\xi_2 = \frac{1}{a}e^{-a|y|}
 \]
 and $a^2=\xi_1^2+|\xi_1|^{\alpha+2}$.
Let us consider the case where $\xi\geq 0$ (the proof works similarly for $\xi<0$).  Assume for the moment that $y\neq 0$. Setting
\begin{align*}
	K_\alpha^+(x,y):=\int_0^\infty \frac{\xi}{(1+\xi^\alpha)^{\frac{1}{2}}}e^{-|y|\xi(1+\xi^\alpha)^{\frac{1}{2}}}e^{\mathrm{i}x\xi}\, \mathrm{d}\xi,
\end{align*}
we can write
\begin{align*}
	K_\alpha^+(x,y)=\int_0^\infty \frac{\xi}{(1+\xi^\alpha)^{\frac{1}{2}}}\frac{1}{G^\prime(\xi)}\frac{d}{d\xi}\left(e^{G(\xi)}\right)\, \mathrm{d}\xi,
\end{align*}
where $G(\xi):=\mathrm{i}x\xi - |y|\xi (1+\xi^\alpha)^\frac{1}{2}$. Using integration by parts, we obtain 
\[
K_\alpha^+(x,y)= -\int_0^\infty \frac{\mathrm{d}}{\mathrm{d}\xi} \left( \frac{\xi}{(1+\xi^\alpha)^{\frac{1}{2}}}\frac{1}{G^\prime(\xi)} \right)e^{G(\xi)}\,\mathrm{d}\xi.
\]
Applying again integration by parts, we find 
\begin{align*}
K_\alpha^+(x,y)
&=-\left[\frac{\mathrm{d}}{\mathrm{d}\xi} \left( \frac{\xi}{(1+\xi^\alpha)^{\frac{1}{2}}}\frac{1}{G^\prime(\xi)} \right)\frac{1}{G^\prime(\xi)} e^{G(\xi)} \right]_0^\infty+\int_0^\infty \frac{\mathrm{d}}{\mathrm{d}\xi}\left(\frac{\mathrm{d}}{\mathrm{d}\xi} \left( \frac{\xi}{(1+\xi^\alpha)^{\frac{1}{2}}}\frac{1}{G^\prime(\xi)} \right)\frac{1}{G^\prime(\xi)}\right)e^{G(\xi)}\,\mathrm{d}\xi\\
&=\frac{\mathrm{d}}{\mathrm{d}\xi} \left( \frac{\xi}{(1+\xi^\alpha)^{\frac{1}{2}}}\frac{1}{G^\prime(\xi)} \right)\frac{1}{G^\prime(\xi)}  \Big|_{\xi=0}+\int_0^\infty \frac{\mathrm{d}}{\mathrm{d}\xi}\left(\frac{\mathrm{d}}{\mathrm{d}\xi} \left( \frac{\xi}{(1+\xi^\alpha)^{\frac{1}{2}}}\frac{1}{G^\prime(\xi)} \right)\frac{1}{G^\prime(\xi)}\right)e^{G(\xi)}\,\mathrm{d}\xi
\end{align*}

In order to lighten the notation, we set
\[
F(\xi):=\frac{\mathrm{d}}{\mathrm{d}\xi} \left( \frac{\xi}{(1+\xi^\alpha)^{\frac{1}{2}}}\frac{1}{G^\prime(\xi)} \right)\frac{1}{G^\prime(\xi)} ,
\]
so that 
\[
K_\alpha^+(x,y)= F(0)+\int_0^\infty F^\prime(\xi)e^{G(\xi)}\,\mathrm{d}\xi.
\]
Using Lemma \ref{lem:F}, we find
\[
|K_\alpha^+(x,y)|\leq \frac{1}{x^2+y^2}.
\]
We are left to consider the case when $y=0$, that is
\[
K_\alpha(x,0)=\int_\R \frac{|\xi|}{(1+|\xi|^\alpha)^\frac{1}{2}}e^{\mathrm{i}x\xi}\,\mathrm{d}\xi.
\]
Notice that $x^2K_\alpha(x,0)=\textcolor{blue}{-}\mathcal{F}^{-1}\left(\frac{\mathrm{d}^2}{\mathrm{d}\xi^2}\frac{|\xi|}{(1+|\xi|^\alpha)^\frac{1}{2}}\right)$ and
\[
\frac{\mathrm{d}^2}{\mathrm{d}\xi^2}\frac{|\xi|}{(1+|\xi|^\alpha)^\frac{1}{2}}=2\delta_0(\xi)+g(\xi),
\]
where $\delta_0$ denotes the delta distribution centered at zero and $g\in L^1(\R)$. Thus $x\mapsto x^2K_\alpha(x,0)$ belongs to $L^\infty(\R)$. This concludes the proof.
\end{proof}

In order to determine the $L^p$-regularity of $K_\alpha$, it is left to investigate the behaviour of the kernel function close to the origin. To do so, we will use that $|\nabla m_\alpha|\lesssim h_\alpha$, where
\begin{equation}\label{eq:h}
h_\alpha(\xi_1,\xi_2):=\frac{\xi_1}{|\xi|^2+\xi_1^{\alpha+2}}.
\end{equation}
Notice also that $\widehat{\partial_x^{-1} K_\alpha}(\xi_1,\xi_2) =-\ii h(\xi_1,\xi_2)$ and \eqref{eq:convolution} can be written as
\begin{equation}\label{eq:H_form}
	\phi=-\frac{\ii}{2}H_\alpha*(\phi^2)_x,\qquad \hat H_\alpha(\xi_1,\xi_2)= h_\alpha(\xi_1,\xi_2)
\end{equation}

\begin{lemma}[The symbol $h_\alpha$]\label{lem:h}
	We have that
	\begin{itemize}
		\item[a)] $h_\alpha\in L^p(\R^2)$ if and only if $\frac{1}{2}+\frac{3}{2(1+\alpha)}<p<2$ and
   \[
H_\alpha\in L^{p^\prime}(\R^2)\qquad \mbox{for} \quad 2<p^\prime < \frac{4+\alpha}{2-\alpha}.
\]
		\item[b)]  $rH_\alpha \in L^\infty(\R^2)$.
		\end{itemize}
\end{lemma}	

\begin{proof}
Similar as in the proof of Lemma \ref{lem:m} we compute
\begin{align*}
\|h_\alpha\|_{L^p(\R^2)}^p &= \int_{\R^2} \frac{1}{\left(|\xi_1|+|\xi_1|^{\alpha+1}+\frac{\xi_2^2}{|\xi_1|}\right)^p}\,\mathrm{d}(\xi_1,\xi_2)\\
&=\int_{\R} \frac{1}{|\xi_1|^p\left(1+|\xi_1|^{\alpha}\right)^p} \int_{\R} \frac{1}{ \left(1+\frac{\xi_2^2}{\xi_1^2(1+|\xi_1|^{\alpha})}\right)^p}\, \mathrm{d}\xi_2 \,\mathrm{d}\xi_1\\
&=\int_{\R} \frac{1}{|\xi_1|^{p-1}\left(1+|\xi_1|^{\alpha}\right)^{p-\frac{1}{2}}} \,\mathrm{d}\xi_1 \int_{\R} \frac{1}{ \left(1+z^2\right)^p}\, \mathrm{d}z,
\end{align*}
where we used the change of variables $z=\frac{\xi_2}{|\xi_1| (1+|\xi_1|^\alpha)^\frac{1}{2}}$. Since the last integral above is bounded for all $p>1$, we find that $h_\alpha\in L^p(\R^2)$ if and only if
\[
\frac{1}{2}+\frac{3}{2(1+\alpha)}<p<2.
\]
Since the Fourier transform is a bounded function from $L^p(\R^2)$ to $L^{p^\prime}(\R^2)$ for $p\in [1,2]$ and $p^\prime$ being the dual conjugate to $p$, we obtain immediately that
\[
H_\alpha\in L^{p^\prime}(\R^2)\qquad \mbox{for} \quad 2<p^\prime < \frac{4+\alpha}{2-\alpha}.
\]
Thereby, part (a) is proved. In order to prove part (b) we proceed as in the proof of Proposition \ref{prop:K_decay}.
We have that
\begin{align*}
H_\alpha(x,y)&=\int_{\R^2}\frac{\xi_1}{\xi_1^2+\xi_2^2+|\xi|^{\alpha+2}}e^{\mathrm{i}x\xi_1+\mathrm{i}y\xi_2}\, \mathrm{d}(\xi_1,\xi_2)\\
&= \int_{\R} \xi_1 e^{\mathrm{i}x\xi_1}\int_{\R} \frac{1}{\xi_1^2+\xi_2^2+|\xi_1|^{\alpha+2}}e^{\mathrm{i}\xi_2y}\,\mathrm{d}\xi_2\, \mathrm{d}\xi_1\\
&= \int_{\R} \frac{\xi}{|\xi|(1+|\xi|^\alpha)^\frac{1}{2} }e^{\mathrm{i}x\xi-|\xi|(1+|\xi|^\alpha)^\frac{1}{2}|y|}\,\mathrm{d}\xi\\
&=\int_{\R} \mbox{sgn}(\xi)\frac{1}{(1+|\xi|^\alpha)^\frac{1}{2} }e^{\mathrm{i}x\xi-|\xi|(1+|\xi|^\alpha)^\frac{1}{2}|y|}\,\mathrm{d}\xi .
\end{align*}
Let us consider the positive part of the integral, the negative part can be estimated analogously. Assume for the moment that $y\neq 0$ and set
\[
H^+_\alpha(x,y):=\int_0^\infty \frac{1}{(1+|\xi|^\alpha)^\frac{1}{2} }e^{\mathrm{i}x\xi-|\xi|(1+|\xi|^\alpha)^\frac{1}{2}|y|}\,\mathrm{d}\xi.
\]
With $E(\xi):= \frac{1}{(1+\xi^\alpha)^\frac{1}{2}}\frac{1}{G^\prime(\xi)}$, we obtain after integration by parts
\[
H^+_\alpha(x,y)= -E(0)- \int_0^\infty E^\prime(\xi)e^{G(\xi)}\, \mathrm{d}\xi,
\]
where $G(\xi)=\mathrm{i}x\xi-|\xi|(1+\xi^\alpha)^\frac{1}{2}|y|$. In view of Lemma \ref{lem:E} we find that
\begin{equation}\label{eq:H_plus}
|H^+_\alpha(x,y)|\lesssim \frac{1}{\sqrt{x^2+y^2}}.
\end{equation}
If $y=0$, we have 
\[
H_\alpha(x,0)=\int_\R \frac{\xi}{|\xi|(1+|\xi|^\alpha)^\frac{1}{2}}e^{\mathrm{i}x\xi}\,\mathrm{d}\xi.
\]
Notice that $\mathrm{i}xH_\alpha(x,y)=-\mathcal{F}^{-1}\left(\frac{\mathrm{d}}{\mathrm{d}\xi}\frac{\xi}{|\xi|(1+|\xi|^\alpha)^\frac{1}{2}}\right) $ and
\[
\frac{\mathrm{d}}{\mathrm{d}\xi}\frac{\xi}{|\xi|(1+|\xi|^\alpha)^\frac{1}{2}} =\frac{1}{(1+|\xi|^\alpha)^\frac{1}{2}} \delta_0(\xi) + g(\xi),
\]
where $\delta_0$ denotes the delta distribution centered at zero and $g\in L^1(\R)$. We deduce that $x\mapsto |x|H_\alpha(x,0)$ is a bounded function. Together with \eqref{eq:H_plus} this proves the claim that $rH_\alpha\in L^\infty(\R^2)$.
\end{proof}

\begin{proposition}\label{prop:K_reg}
The kernel function $K_\alpha$ satisfies the regularity
    \[
K_\alpha \in L^{r}(\R^2)\qquad \mbox{for}\quad 1<r< \frac{8+2\alpha}{8-\alpha}.
\]
\end{proposition}

\begin{proof}
We know already from Lemma \ref{lem:K_smoothness} and Proposition \ref{prop:K_decay} that $K_\alpha$ is smooth outside the origin and $r^2K_\alpha\in L^\infty(\R^2)$. Introducing a smooth truncation function $\rho:\R_+\to \R_+$, which is compactly supported in a neighborhood of zero, denoted by $B\subset \R^2$, with $\rho(0)=1$, we find that
\begin{equation}\label{eq:K_inf}
(1-\rho(r))K_\alpha \in L^s(\R^2)\qquad \mbox{for all}\quad s>1, 
\end{equation}
where $r(x,y)=|(x,y)|$.
In order to determine the regularity of $K_\alpha$ close to zero, recall that $|\nabla m_\alpha|\lesssim |h_\alpha|$, where $h_\alpha$ defined in \eqref{eq:h} so that
\begin{equation*}
|\nabla m_\alpha |\in L^q(\R^2)\qquad \mbox{for}\qquad \frac{1}{2}+\frac{3}{2(1+\alpha)}<q<2,
\end{equation*}
due to Lemma \ref{lem:h} (a). Now, we use that the Fourier transformation is a bounded operator from $L^p(\R^2)$ to $L^{p^\prime}(\R^2)$ when $p\in [1,2]$ and $p^\prime$ is the dual conjugate of $p$ and obtain that
\[
\|rK_\alpha\|_{L^{q^\prime}(\R^2)}\lesssim \|xK\|_{L^{q^\prime}(\R^2)}+\|yK\|_{L^{q^\prime}(\R^2)} \leq  \|\partial_{\xi_1}m_\alpha\|_{L^{q}(\R^2)}+\|\partial_{\xi_2}m_\alpha\|_{L^{q}(\R^2)}\lesssim  2\||\nabla m_\alpha|\|_{L^{q}(\R^2)},
\]
so that
\[
rK_\alpha \in L^{q^\prime}(\R^2)\qquad \mbox{for}\qquad 2<q^\prime < \frac{4+\alpha}{2-\alpha},
\]
and in fact $rK_\alpha \in L^{s}(B)$ for $1\leq s<\frac{4+\alpha}{2-\alpha}$, by Hölder's inequality and the boundedness of $B$. Then, we estimate
\[
\|\rho(r)K_\alpha\|_{L^r(\R^2)}\lesssim \|r^{-1}\| _{L^t(B)} \|rK_\alpha\|_{L^{s}(B)},
\]
where $\frac{1}{r}=\frac{1}{t}+\frac{1}{s}$. Since $r^{-1}\in L^t(B)$ if and only if $1\leq t<2$, we conclude that
\[
\rho(r)K_\alpha \in L^r(\R^2)\qquad \mbox{for} \quad 1\leq r<\frac{8+2\alpha}{8-\alpha}.
\]
In view of \eqref{eq:K_inf} we deduce that
\[
K_\alpha \in L^{r}(\R^2)\qquad \mbox{for}\quad 1<r< \frac{8+2\alpha}{8-\alpha}.
\]
\end{proof}

\subsubsection*{A non-optimal decay rate.}
First notice that $\phi$ inherits the integrability properties of $K_\alpha$, since
\[
\|\phi\|_{L^r(\R^2)}= \|K_\alpha*\phi^2\|_{L^r(\R^2)} \lesssim \|K_\alpha\|_{L^r(\R^2)}\|\phi^2\|_{L^1(\R^2)}\lesssim \|K_\alpha\|_{L^r(\R^2)},
\]
by the $L^2$-integrability of $\phi$. Interpolating between the boundedness of $\phi$, which is due to Lemma \ref{lem:smooth}, and the $L^r$-integrability of $\phi$ for $1<r<\frac{8+2\alpha}{8-\alpha}$, we actually find that
\begin{equation}\label{eq:phi_int}
\phi \in L^p(\R^2)\qquad \mbox{for all}\quad p>1.
\end{equation}

\begin{proposition}[A priori decay estimate]\label{prop:non_op}
	If $\phi$ is a solution of \eqref{eq:eq} in the energy space $X_{\frac{\alpha}{2}}$, then
	\[
	r\phi \in L^\infty(\R^2).
	\]
	
\end{proposition}

\begin{proof} Recall from \eqref{eq:H_form} that
\[
\phi =-\frac{\mathrm{i}}{2}H_\alpha *(\phi^2)_x
\]
so that by Young's inequality
\begin{align*}
\|r\phi\|_\infty &\leq \|rH*\phi \phi_x\|_\infty + \|H*r\phi\phi_x\|_\infty\lesssim \|rH\|_\infty\|\phi\|_{L^2(\R)}\|\phi_x\|_{L^2(\R)} + \|H\|_{L^{q^\prime}(\R^2)}\|r\phi_x\|_{L^2(\R^2)} \|\phi\|_{L^s(\R^2)},
\end{align*}
where $\frac{1}{s}+\frac{1}{2}=\frac{1}{q}$ for $\frac{1}{2}+\frac{3}{2(1+\alpha)}<q<2$ and $q^\prime$ being the dual of $q$. 
Now, the statement follows from Lemma \ref{lem:h} and Lemma \ref{lem:est}.
\end{proof}

\begin{proposition}[A non-optimal decay rate]\label{prop:non_op_2}
If $\phi$ is a solitary solution of \eqref{eq:eq}, then
\[
r^{1+\delta} \phi \in L^\infty(\R^2)
\]
for any $0\leq \delta <1$.
 \end{proposition}

 \begin{proof}
 We use the regularity in Proposition \ref{prop:non_op} to improve the decay rate by estimating
 \[
\|r^{1+\delta}\phi\|_\infty \lesssim \|r^{1+\delta}K_\alpha *\phi^2\|_\infty + \|K_\alpha * r^{1+\delta}\phi^2\|_\infty,
 \]
 where we also used the convexity of $r^{1+\delta}$. The first norm on the right-hand side above is clearly bounded by Young's inequality, Proposition \ref{prop:K_decay} and the $L_2$-integrability of $\phi$. For the second norm, let $\e>0$ be a small constant so that $0<\delta<\frac{1}{1+\e}<1$. Using that $K_\alpha \in L^{1+\e}(\R^2)$ for $\e>0$ small enough,  we estimate
 \[
\|K_\alpha * r^{1+\delta}\phi^2\|_\infty \lesssim \|K_\alpha\|_{L^{1+\e}(\R^2)}\|r^{1+\delta}\phi^2\|_{L^{\frac{1+\e}{\e}}(\R^2)}.
 \]
 Notice that
\[
\|r^{1+\delta}\phi^2\|_{L^{\frac{1+\e}{\e}}(\R^2)}^{\frac{\e}{1+\e}} = \int_{\R^2} |r\phi|^{\frac{(1+\delta)(1+\e)}{\e}} |\phi|^{(1-\delta)\frac{1+\e}{\e}}\, \mathrm{d}(x,y)\leq \|r\phi\|_\infty^{\frac{(1+\delta)(1+\e)}{\e}} \int_{\R^2} |\phi|^{(1-\delta)\frac{1+\e}{\e}}\, \mathrm{d}(x,y).
\]
By our choice of $\e>0$, we have $(1-\delta)\frac{1+\e}{\e}>1$, so  the above norm is bounded by \eqref{eq:phi_int}, which concludes the proof of the statement.
\end{proof}

\subsubsection*{Proof of Theorem \ref{thm:decay}}

In view of the discussion at the beginning of this section and Lemma \ref{lem:smooth}, we obtain our main result 
\[
r^2\phi \in L^\infty(\R^2)
\]
provided that (A) and (B) at the beginning of the section are satisfied.
The statement in (A) is proved in Proposition \ref{prop:K_decay}, while the first part of statement (B) follows from Proposition \ref{prop:K_reg}, where it is shown that
\[
K_\alpha \in L^{r}(\R^2)\qquad \mbox{for}\quad 1<r< \frac{8+2\alpha}{8-\alpha}.
\]
Now, we make use of the non-optimal decay estimate in Proposition \ref{prop:non_op_2} to show that indeed $r^2\phi^2 \in L^{r^\prime}(\R^2)$, where $r^\prime$ is the dual conjugate to $r$. For any $0\leq \delta<1$, we have that

\[
\int_{\R^2} |r^2\phi^2|^{r^\prime}\,\mathrm{d}(x,y) \leq \|r^{1+\delta}\phi\|_\infty^{\frac{2r^\prime}{1+\delta}}\int_{\R^2}  \phi^{2r^\prime\frac{\delta}{1+\delta}}\,\mathrm{d}(x,y).
\]
Choosing $\delta =r-1 \in (0,1)$ we find that
$
\frac{\delta}{1+\delta}r^\prime = \frac{r-1}{r} r^\prime =1
$ and the boundedness of $r^2\phi^2$ in $L^{r^\prime}(\R^2)$ follows from the $L^2$-integrability of $\phi$. Hence, statement (B) is shown, which concludes the proof of Theorem \ref{thm:decay}. \hfill\qedsymbol{} \\

To visualize the decay rate of the solutions we consider the product of numerically generated lump solutions with $1/r^2(x,y)=1/(x^2+y^2)$. Figure \ref{fig:decay} shows $x$ and $y-$cross sections of this product for $\alpha=2$, $\alpha=1.7$, and $\alpha=1.35$. As the decay rate is quadratic the result approaches to a constant value for increasing $|x|$ and $|y|$, as expected. We observe that the behaviour is similar for all $\alpha$ values  but the aforementioned constant becomes smaller for smaller values of $\alpha$. 
\begin{figure}[H]
	\begin{minipage}[t]{0.45\linewidth}
		\includegraphics[width=3.1in,height=2.3in]{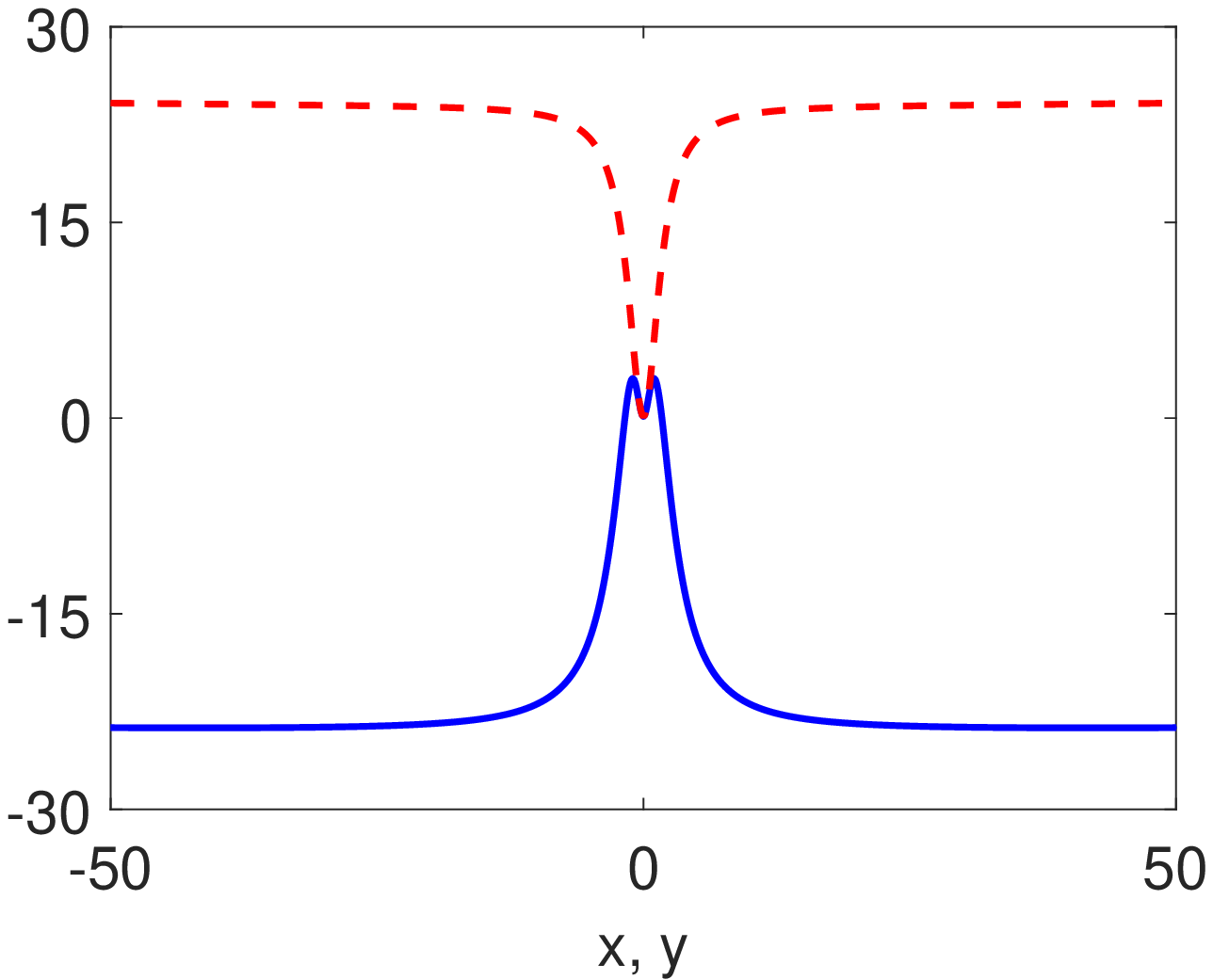}
	\end{minipage}
	\hspace{30pt}
	\begin{minipage}[t]{0.45\linewidth}
		\includegraphics[width=3.1in,height=2.3in]{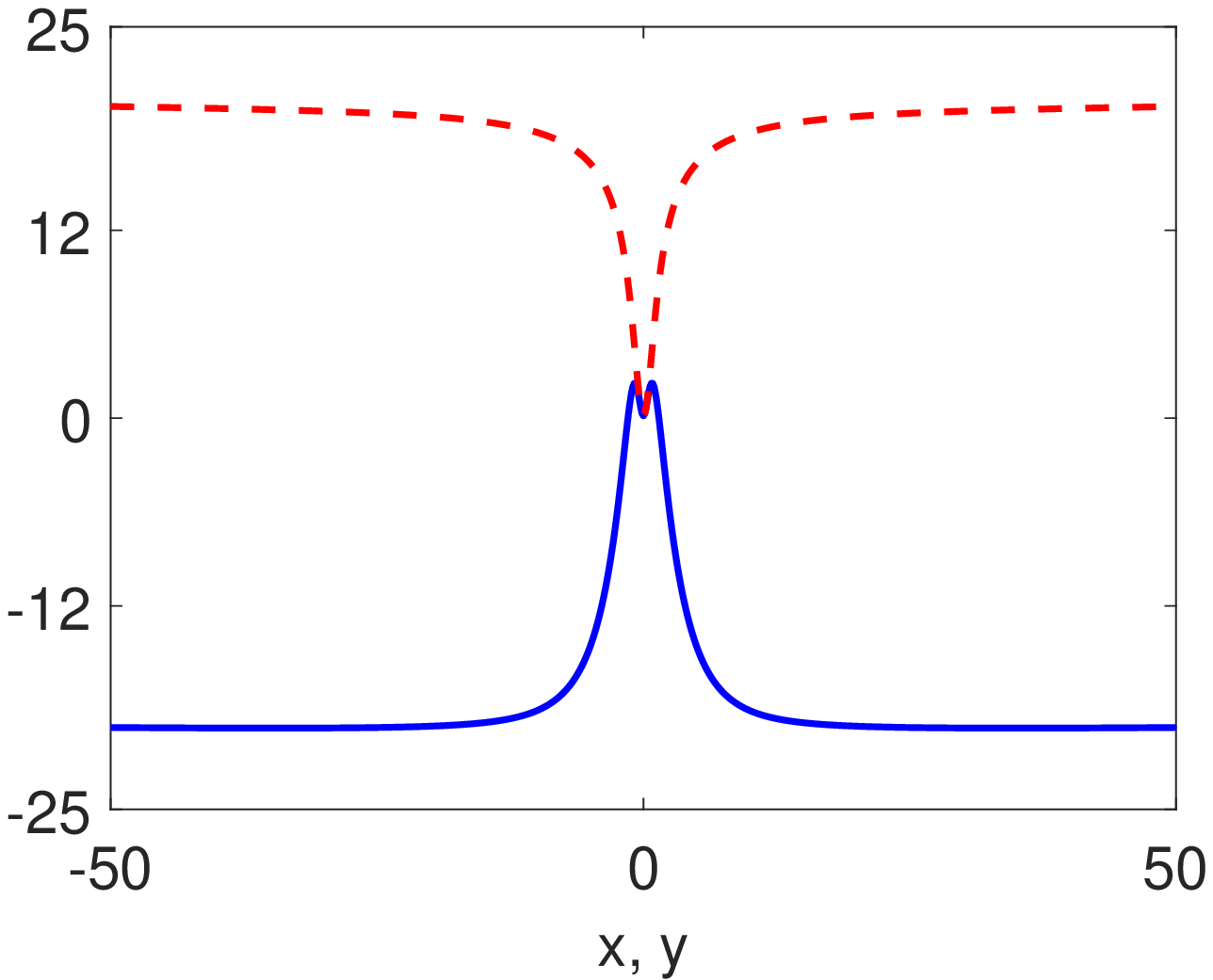}
	\end{minipage}
 \vspace{10pt}
\hspace{5cm}
	\begin{minipage}[t!]{0.45\linewidth}
		\includegraphics[width=3.1in,height=2.3in]{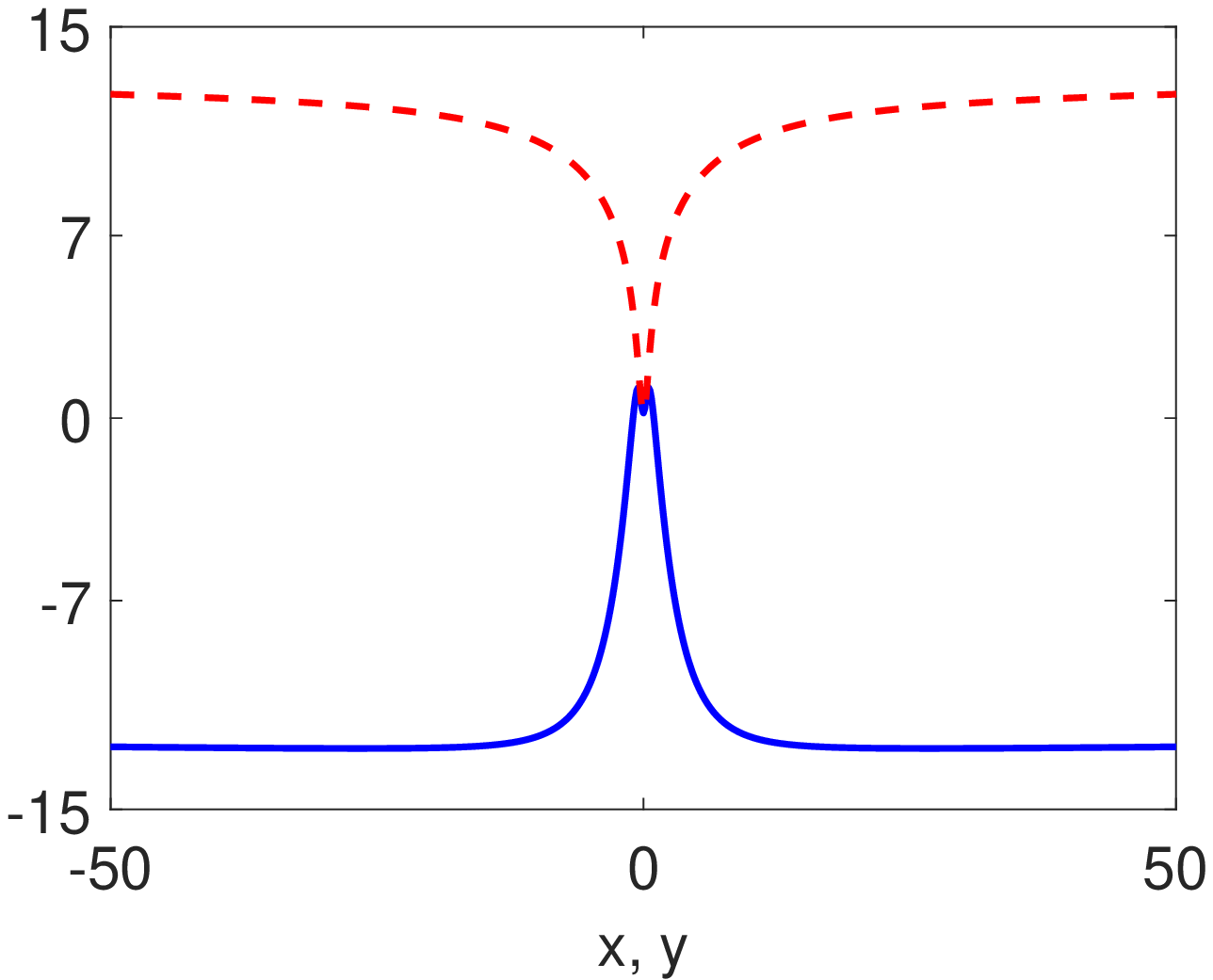}
	\end{minipage}
\caption{The  $x-$cross section (solid line) and the $y-$cross section (dashed line) of the  product of numerically generated lump solutions with  $1/r^2(x,y)=1/(x^2+y^2)$ for  $\alpha=2$ (top left panel), $\alpha=1.7$ (top right panel) and $\alpha=1.35$ (bottom panel).}
	\label{fig:decay}
\end{figure}

\medskip

\begin{remark}[Benjamin--Bona--Mahony KP equation]
\emph{We would like to point out that the decay result in Theorem \ref{thm:decay} is equally valid for lump solutions of the fractional BBM-KP equation, which  is when the term $\mathrm{D}_x^\alpha u_x$ in \eqref{eq:fKP} is replaced by $\mathrm{D}^\alpha _xu_t$.
}
\end{remark}

\begin{remark}[Rotation modified KP equation]
\emph{
Lump solutions $u(t,x,y)=\phi(x-ct,y)$ of the \emph{rotation modified KP equation} 
\[
(u_t + uu_x - \beta u_{xxx})_x+u_{yy}=\gamma u,
\]
where $\beta\in \R$ determines the type of dispersion and $\gamma>0$ is the Coriolis parameter due to the Earth's rotation,
exist for $\beta>0$ and $c<2\sqrt{\gamma\beta}$, cf. \cite[Theorem 2.2, Remark 2.4]{Chen0}. They satisfy the convolution equation
\[
\phi = -\frac{1}{2}K*\phi^2, \qquad \hat{K}(\xi_1,\xi_2)=m(\xi_1,\xi_2),
\]
where 
\[m(\xi_1,\xi_2)=\frac{\xi_1^2}{-c\xi_1^2+\beta \xi_1^4+\xi_2^2+\gamma}.
\]
Due to the Coriolis parameter $\gamma>0$, the symbol $m$ is \textit{smooth} at the origin, which allows lump solutions to decay exponentially at infinity (cf. \cite[Theorem 1.6]{Chen1}). If the dispersive term $\beta u_{xxx}$ were replaced by the fractional term $-\beta |\mathrm{D}_x|^\alpha u_x$, which would lead to an \emph{fractional rotation modified KP equation}, we'd expect a decay of lump solutions, which depends on $\alpha$ (in a similar way as we see it for the fractional KdV equation \cite{Eychenne, Kenig}).
}
\end{remark}

\begin{remark}[Full dispersion KP equation]
\emph{
The \emph{full dispersion KP equation} is given by
\[
u_t - l(\mathrm{D})u_x + uu_x=0,
\]
where
\begin{equation*}
l(\mathrm{D})=(1+\beta |\mathrm{D}|^2)^\frac{1}{2}\left(\frac{\tanh(|\mathrm{D}|)}{|\mathrm{D}|}\right)^\frac{1}{2} \left(1+\frac{\mathrm{D}_y^2}{\mathrm{D}_x^2}\right)^\frac{1}{2}.
\end{equation*}
Here $\beta>0$ is the surface tension coefficient. 
Existence of lump solutions for full dispersion KP equation  is shown in \cite{eg, egn}.
In the same way as for the fKP-I equation, the transverse direction induces a discontinuity at the origin of the symbol $m(\xi_1,\xi_2)= \frac{1}{c+l(\xi_1,\xi_2)}$. Therefore, the decay of lump solutions will also be at most quadratically.
}
\end{remark}

\bigskip
	
	\appendix
	
	\section{Auxiliary results}

 \begin{lemma}[Fractional integration by parts] \label{lem:frac_int}
 Let $\alpha\geq 0$. Then,
 \begin{align*}
\int_{\R} \phi \mathrm{D}^\alpha_x \phi \,\mathrm{d}x = \int_{\R^2}\left(\mathrm{D}^{\frac{\alpha}{2}}\phi\right)^2\,\mathrm{d}x, \qquad \int_{\R}x \phi_x \mathrm{D}^\alpha_x \phi \, dx=\frac{\alpha-1}{2}\int_{\R} \left( \mathrm{D}^\frac{\alpha}{2}_x\phi\right)^2\, \mathrm{d}x 
\end{align*}
and
\begin{align*}
\int_{\R} x^2\phi_{xx} \mathrm{D}^\alpha_x \phi \,\mathrm{d}x = -\int_{\R} x^2 \left( \mathrm{D}^\frac{\alpha}{2}_x\phi_x\right)^2\,\mathrm{d}x + \tfrac{1}{4}(\alpha-2)^2 \int_{\R}\left( \mathrm{D}^\frac{\alpha}{2}_x\phi\right)^2\,\mathrm{d}x.
\end{align*}

 \end{lemma}

 \begin{proof}
 The first assertion follows immediately by
 \[
\int_{\R} \phi \mathrm{D}^\alpha_x \phi \,\mathrm{d}x = \langle \phi, \overline{\mathrm{D}^\alpha_x \phi }\rangle = \langle \hat \phi, |\xi|^\alpha \overline{\hat \phi }\rangle =\langle |\xi|^\frac{\alpha}{2}\hat \phi, \overline{|\xi|^\frac{\alpha}{2} \hat \phi }\rangle = \int_{\R}\left(\mathrm{D}^\frac{\alpha}{2}_x\phi\right)^2\,\mathrm{d}x. \]
For the second statement, notice first that
\[
\langle \hat \phi _\xi, \overline{|\xi|^\alpha \xi \hat \phi} \rangle=- \langle \hat \phi, (\alpha+1) |\xi|^\alpha \overline{\hat \phi }+ |\xi|^\alpha \xi \overline{\hat \phi_\xi }\rangle = - (\alpha+1)\int_{\R^2}\left(\mathrm{D}^\frac{\alpha}{2}_x\phi\right)^2\, \mathrm{d}x - \langle \hat \phi_\xi, \overline{|\xi|^\alpha \xi\hat \phi }\rangle,
\]
therefore
\begin{equation}\label{eq:1}
\langle \hat \phi _\xi,\overline{ |\xi|^\alpha \xi \hat \phi }\rangle=-\tfrac{\alpha+1}{2} \int_{\R}\left(\mathrm{D}^\frac{\alpha}{2}_x\phi\right)^2\, \mathrm{d}x,
\end{equation}
which implies 
\[
\int_{\R} x \phi_x \mathrm{D}^\alpha_x \phi \,\mathrm{d}x = -\langle  (\xi \hat \phi)_\xi, |\xi|^\alpha \overline{\hat \phi} \rangle = - \langle \hat \phi, |\xi|^\alpha \overline{\hat \phi} \rangle - \langle \hat \phi_\xi, |\xi|^\alpha \xi \overline{\hat \phi} \rangle =-\int_{\R} \left( \mathrm{D}^\frac{\alpha}{2}_x\phi\right)^2\, \mathrm{d}x + \frac{\alpha+1}{2}\int_{\R}\left(\mathrm{D}^\frac{\alpha}{2}_x\phi\right)^2\, \mathrm{d}x.
\]
Turning to the third statement, notice first that
\begin{align}\label{eq:2}
\int_{\R}x^2 \left(\mathrm{D}^{\frac{\alpha}{2}}_x\phi\right)^2\,\mathrm{d}x&=-\langle (|\xi|^\frac{\alpha}{2}\xi \hat \phi)_{\xi\xi}, |\xi|^\frac{\alpha}{2}\xi \overline{\hat \phi}\rangle= -\left(\tfrac{\alpha}{2}+1\right)\tfrac{\alpha}{2}\int_{\R}\left(\mathrm{D}^\frac{\alpha}{2}_x\phi\right)^2\, \mathrm{d}x+ \langle \hat \phi_\xi, |\xi|^\alpha \xi^2 \overline{\hat \phi_\xi} \rangle,
\end{align}
where we used \eqref{eq:1}.
Now,
\begin{align*}
\int_{\R} x^2\phi_{xx} \mathrm{D}^\alpha_x \phi \,\mathrm{d}x  &= \langle (\xi^2 \hat \phi )_{\xi\xi}, |\xi|^\alpha \hat \phi \rangle \\
&= 2\langle \hat \phi , |\xi|^\alpha \overline{\hat \phi} \rangle + 4\langle \hat \phi_\xi , |\xi|^\alpha \xi \overline{\hat \phi} \rangle+ \langle\hat \phi_{\xi \xi}, |\xi|^\alpha  \xi^2\overline{\hat \phi} \rangle\\
&= 2\langle \hat \phi , |\xi|^\alpha \overline{\hat \phi} \rangle + (2-\alpha)\langle \hat \phi_\xi , |\xi|^\alpha \xi \overline{\hat \phi} \rangle- \langle \hat \phi_{\xi}, |\xi|^\alpha  \xi^2\overline{\hat \phi_\xi} \rangle\\ 
&= -\int_{\R} x^2 \left( \mathrm{D}^\frac{\alpha}{2}_x\phi_x\right)^2\,\mathrm{d}x + \tfrac{1}{4}(\alpha-2)^2 \int_{\R}\left( \mathrm{D}^\frac{\alpha}{2}_x\phi\right)^2\,\mathrm{d}x,
\end{align*}
by \eqref{eq:1} and \eqref{eq:2}.
\end{proof}
	
	\begin{lemma}[Properties of $F$]\label{lem:F}
		Let $\alpha >0$, $G(\xi)=\mathrm{i}x\xi - |y|\xi (1+\xi^\alpha)^\frac{1}{2}$ for $\xi\geq 0$ and $y\neq 0$. The function \[
  F(\xi)=\frac{\mathrm{d}}{\mathrm{d}\xi} \left( \frac{\xi}{(1+\xi^\alpha)^{\frac{1}{2}}}\frac{1}{G^\prime(\xi)} \right)\frac{1}{G^\prime(\xi)}\qquad \mbox{for}\quad \xi\geq 0
  \]
  satisfies
		\begin{itemize}
			\item[(a)] $F(0)=\frac{1}{[G^\prime(0)]^2}=(\mathrm{i}x-|y|)^{-2}$
			\item[(b)] 
			
			$|F^\prime(\xi)|\lesssim T(\xi) \frac{1}{x^2+y^2}$, for some function $T$ such that $Te^{G}\in L^1(\R_+).$
			
		\end{itemize}
	\end{lemma}

\begin{proof} Let us first summarize all needed derivatives for the function $G$:
	\begin{align*}
		G(\xi)&=\mathrm{i}x\xi - |y|\xi (1+\xi^\alpha)^\frac{1}{2}\\
		G^\prime(\xi)&= \mathrm{i}x -|y|\left((1+\xi^\alpha)^\frac{1}{2} + \frac{\alpha}{2}\xi^\alpha(1+\xi^\alpha)^{-\frac{1}{2}}\right)
		= \frac{1}{2}(1+\xi^\alpha)^{-\frac{1}{2}} \left(2\mathrm{i}x (1+\xi^\alpha)^{\frac{1}{2}}- |y|\left(2+ (2+\alpha)\xi^\alpha\right)\right)\\
		G^{\prime\prime}(\xi)&=-\frac{\alpha}{4}|y|(1+\xi^\alpha)^{-\frac{3}{2}}\xi^{\alpha-1}\left(2(1+\alpha)+(2+\alpha)\xi^\alpha\right)\\
		G^{\prime\prime\prime}(\xi)&=  \frac{\alpha}{8}|y|(1+\xi^\alpha)^{-\frac{5}{2}}\xi^{\alpha-2}\left(4(1+\xi^\alpha)^2-\alpha^2(\xi^{2\alpha}+2\xi^\alpha+4)\right)
	\end{align*}
Then, we compute $F$ as
\begin{align*}
F(\xi)=	\frac{1}{(1+\xi^\alpha)^\frac{1}{2}[G^\prime(\xi)]^2} - \frac{\xi \left(\frac{\alpha}{2}(1+\xi^\alpha)^{-\frac{1}{2}}\xi^{\alpha-1}G^\prime(\xi) +(1+\xi^\alpha)^{\frac{1}{2}}G^{\prime\prime}(\xi) \right)}{(1+\xi^\alpha)\left[G^\prime(\xi)\right]^3},
\end{align*}
which yields $F(0)=\frac{1}{[G^\prime(0)]^2}=(\mathrm{i}x-|y|)^{-2}$ and proves part (a).
A tedious, but straightforward computation yields that the  derivative $F$ is given by 
\begin{align*}
F'(\xi)&=-\frac{\alpha}{2}\frac{(1+\alpha)\xi^{\alpha-1}+(1-\frac{\alpha}{2})\xi^{2\alpha-1}}{(1+\xi^\alpha)^\frac{5}{2} G'(\xi)^2}+\frac{(\frac{3}{2}\alpha\xi^\alpha-3(1+\xi^\alpha))G''(\xi)-(1+\xi^\alpha)\xi G'''(\xi)}{(1+\xi^\alpha)^\frac{3}{2}G'(\xi)^3}+\frac{3\xi G''(\xi)^2}{(1+\xi^\alpha)^\frac{1}{2}G'(\xi)^4}   \\
&=: T_1(\xi)+T_2(\xi)+T_3(\xi)
\end{align*}

Now, we insert the expressions for $G^\prime, G^{\prime\prime},$ and $G^{\prime\prime\prime}$. We have
\begin{align*}
|G^\prime(\xi)|^2 &=\frac{1}{4}(1+\xi^\alpha)^{-1}\left(4x^2(1+\xi^\alpha)+y^2(2+(2+\alpha)\xi^\alpha)^2\right)\gtrsim x^2+y^2,\\
|G^{\prime\prime}(\xi)|&\eqsim |y|(1+\xi^\alpha)^{-\frac{1}{2}}\xi^{\alpha-1},\\
|G^{\prime\prime\prime}(\xi)|&\eqsim  |y|(1+\xi^\alpha)^{-\frac{5}{2}}\xi^{\alpha-2}\left|(4-\alpha^2)(1+\xi^\alpha)^2-3\alpha^2\right|.
\end{align*}
Starting with $T_1$ we estimate
\begin{align}\label{eq:T1}
	|T_1(\xi)| \lesssim \frac{\xi^{\alpha-1}}{(1+\xi^\alpha)^\frac{3}{2}(x^2+y^2)}.
\end{align}
For $T_2$ we find 
\begin{align}\label{eq:T_2}
	|T_2(\xi)|\lesssim |y|\frac{ \xi^{\alpha-1}}{(1+\xi^\alpha)(x^2+y^2)^\frac{3}{2}}.
\end{align}
Eventually, we estimate $T_3$ as
\begin{align}\label{eq:T_3}
|T_3(\xi)|\lesssim y^2	\frac{\xi^{2\alpha-1}}{(1+\xi^\alpha)^\frac{3}{2}(x^2+y^2)^2}.
\end{align}
Summarizing \eqref{eq:T1}-\eqref{eq:T_3}, we find that $|F^\prime(\xi)|\lesssim T(\xi) \frac{1}{x^2+y^2}$, where $Te^{G}\in L^1(\R_+)$ which proves part (b).
\medskip

\end{proof}

\begin{lemma}[Properties of $E$]\label{lem:E}
Let $\alpha>0$, $G(\xi)=\mathrm{i}x\xi - |y|\xi (1+\xi^\alpha)^\frac{1}{2}$ for $\xi\geq 0$  and $y\neq 0$. The function 
\[
E(\xi)=\frac{1}{(1+\xi^\alpha)}\frac{1}{G^\prime(\xi)}\qquad \mbox{for}\quad \xi\geq 0
\]
satisfies
\begin{itemize}
\item[(a)] $E(0)=\frac{1}{G^\prime(0)}=\mathrm{i}x-|y|$
\item[(b)] $|E^\prime(\xi)|\lesssim S(\xi) \frac{1}{x^2+y^2}$, for some function $S$ such that $Se^{G}\in L^1(\R_+)$.
\end{itemize}
\end{lemma}

\begin{proof}
The proof follows by direct computation. Recall that 
\begin{equation}\label{eq:E}
|G^\prime(\xi)|^2 =\frac{1}{4}(1+\xi^\alpha)^{-1}\left(4x^2(1+\xi^\alpha)+y^2(2+(2+\alpha)\xi^\alpha)^2\right)\gtrsim x^2+y^2
\end{equation}
Part (a) follows immediately from $G^\prime(0)=\mathrm{i}x-|y|$. For part (b), we compute the derivative of $E^\prime$ and use \eqref{eq:E} to estimate
\[
E^\prime(\xi)=\frac{\alpha}{4} \frac{\xi^{\alpha-1}}{(1+\xi^\alpha)^{2}[G^\prime(\xi)]^2}\lesssim \frac{\xi^{\alpha-1}}{(1+\xi^\alpha)^2}\frac{1}{x^2+y^2},
\]
which yields the statement.
  \end{proof}
\bigskip

\noindent \textbf{Acknowledgements}\\
The authors would like to thank Christian Klein and Dmitry E. Pelinovsky for their helpful discussions. This research was carried out while D.N. was supported by the Wallenberg foundation.


\begin{thebibliography}{10}

\bibitem{albert1}
Albert, J.P., (1999), \emph{Concentration compactness and the stability of solitary-wave
solutions to nonlocal equations}. Contemp. Math., 221:1–30.


\bibitem{amaral}
Amaral, S., Borluk, H., Muslu, G.M., Natali, F. and Oruc, G., (2022), \emph{On the existence and spectral stability of periodic waves for the fractional Benjamin--Bona--Mahony equation}. Stud. Appl. Math., 148(1), 62--98.






\bibitem{BonaLi}
Bona, J.L. and Li, Y.A., (1997),
\emph{Decay and analyticity of solitary waves}.
J. Math. Pures Appl., 76, 377--430.

\bibitem{BBN}
Borluk, H., Bruell, G. and Nilsson, D., (2022), 
\emph{Traveling waves and transverse instability for the fractional Kadomtsev-Petviashvili equation}.
 Stud. Appl. Math., 149 (1), 95–123.

\bibitem{bgsw}
Buffoni, B., Groves, M. D., Sun, S. M. and Wahl\'{e}n, E., (2013), \emph{Fully localised solitary-wave solutions of the three-dimensional gravity-capillary water-wave problem.} J. Differ. Equ., 254(3), 1006--1096.

\bibitem{bgw}
Buffoni, B., Groves, M. D. and Wahl\'{e}n, E., (2022), \emph{Fully localised three-dimensional gravity-capillary solitary waves on water of infinite depth.} J. Math. Fluid Mech., 24(2) Paper No. 55, 21. 

\bibitem{Chen0}
Chen, R.M., Hur, V.M. and Liu, Y., (2008),
\emph{Solitary waves of the rotation-modified Kadomtsev--Petviashvili equation}. Nonlinearity, 21, 2949--2979.

\bibitem{Chen1}
Chen, R.M., Liu, Y. and Zhand, P., (2012),
\emph{Local regularity and decay estimates of solitary waves for the rotation-modified Kadomtsev--Petviashvili equation}.
Transactions of the American Math. Society 364(7), 3395--3425.



\bibitem{deBouard0}
de Bouard, A. and Saut, J.-C., (1997), \emph{Solitary waves of generalized Kadomtsev-Petviashvili equations.} Ann. Inst. H. Poincar\'{e} C Anal. Non Lin\'{e}aire., 14, 211--236.

\bibitem{deBouard}
de Bouard, A. and Saut, J.-C., (1997), \emph{Symmetries and decay of the generalized Kadomtsev--Petviashvili solitary waves.} SIAM J. Math. Anal., 28 (5), 1064--1085.

\bibitem{eg}
Ehrnstr\"{o}m, M. and Groves, M. D., (2018), \emph{Small-amplitude fully localised solitary waves for the full-dispersion {K}adomtsev-{P}etviashvili equation.} Nonlinearity, 31(12), 5351--5384.

\bibitem{egn}
Ehrnstr\"{o}m, M. and Groves, M. D. and Nilsson, D., (2022), \emph{Existence of {D}avey-{S}tewartson type solitary waves for the fully dispersive {K}adomtsev-{P}etviashvilii equation.} SIAM J. Math. Anal., 54(4), 4954--4986.

\bibitem{Esfahani}
Esfahani, A. (2015), \emph{Anisotropic Gagliardo--Nirenberg inequality with fractional derivatives}. Z. Angew. Math. Phys., \textbf{66}, 3345--3356.


\bibitem{Eychenne}
Eychenne, A. and Valet, F., (2012), 
\emph{Decay of solitary waves of fractional Korteweg--de Vries type equations}. J. Differ. Equ., 363, 243–274.

\bibitem{fonseca}
Fonseca, G., Linares, F. and Ponce, G., (2013), \emph{The IVP for the dispersion generalized
Benjamin--Ono equation in weighted Sobolev spaces}. Ann. I. H. Poincar\'e-AN., 30(5):763–790.

\bibitem{franklenzmann}
Frank, R.L., and  Lenzmann, E., (2013), \emph{Uniqueness of non-linear ground states for fractional Laplacians in ${\mathbb {R}} $.} Acta Math., 210(2), 261–318.

\bibitem{Grafakos}
Grafakos, L.,  (2014), \emph{{Modern {F}ourier Analysis}}. Graduate Texts in Mathematics. Third edition, Springer, New York.


\bibitem{gs}
Groves, M.D. and Sun, S. M., (2008), \emph{Fully localised solitary-wave solutions of the three-dimensional gravity-capillary water-wave problem}. Arch. Ration. Mech. Anal. 18, 1-91. 













\bibitem{kp}
Kadomtsev, B.B. and Petviashvili, V.I., (1970), \emph{On the stability of solitary waves in a weakly dispersing medium}. Sov. Phys. Dokl., 15, 539--541.



\bibitem{Kenig}
Kenig, C.E., Martel, Y. and Robbiano, L., (2011),
\emph{Local well-posedness and blow-up in the energy space for a class of $L^2$ critical dispersion generalized Benjamin–-Ono equations.} Ann. Inst. H. Poincare Anal. Non Lineaire 28, 853-–887.

\bibitem{klein}
Klein, C. and Saut, J-C., (2015),  \emph{A numerical approach to blow-up issues for dispersive perturbations of Burgers equation.} Physica D, 295, 46–65.

\bibitem{klein1}
Klein, C. and Saut, J-C., (2012),  \emph{Numerical study of blow-up and stability of solutions of generalized Kadomtsev-Petviashvili Equations.} J. Nonlinear Sci., 22:4763--811.


\bibitem{klein2}
Klein, C., Sparber, C. and Markowich, P., (2007),  \emph{Numerical study of oscillatory regimes
in the Kadomtsev–Petviashvili equation.} J. Nonlinear Sci., 17:429 – 470.


\bibitem{lannes} 
Lannes, D., (2013), \emph{The water waves problem.} Mathematical Surveys and Monographs, 188, American Mathematical Society, Providence, RI.

\bibitem{linares1}
Linares, F., Pilod, D. and Saut, J. C., (2018), \emph{The Cauchy problem for the fractional Kadomtsev--Petviashvili equations.} SIAM J.  Math. Anal., 50(3), 3172--3209.

\bibitem{linares2}
Linares,  F., Pilod, D.  and Saut, J-C., (2015), \emph{Remarks on the orbital stability of ground
state solutions of fKdV and related equations}. Adv Differential Equ., 20(9-10), 835--858.


\bibitem{Lizorkin}
Lizorkin, P.I., (1967), 
\emph{Multipliers of Fourier integrals}.
Proc. Steklov Inst. Math., \textbf{89}, 269--290.




\bibitem{MS}
Miloh, T. and Spector, M.D., (1994), \emph{Stability of nonlinear periodic internal waves in a deep stratified fluid}.
SIAM J. Appl. Math. 54(3), 688--707.







\bibitem{natali}
Natali, F., Le U.  and  Pelinovsky, D.E., (2020), \emph{New variational characterization of periodic waves in the fractional Korteweg-de
Vries equation.} Nonlinearity,  33, 1956–1986.

\bibitem{oruc}
Oruc, G., Borluk, H. and Muslu, G.M., (2020),
\emph{The generalized fractional Benjamin--Bona--Mahony equation: Analytical and numerical results}. Physica D, 132499.

\bibitem{pava}
Pava, J.A., (2018),  \emph{Stability properties of solitary waves for fractional KdV and
BBM equations.}  Nonlinearity.,  31(3), 920–956.

\bibitem{pelinovski}
Pelinovsky, D.E. and Stepanyants, Y.A., (2004), \emph{Convergence of Petviashvili’s iteration method for numerical approximation of stationary solutions of nonlinear
wave equations.} SIAM J Numer Anal., 42:1110–1127.

\bibitem{pelinovski2}
Le, U. and Pelinovsky, D.E., (2019), \emph{Convergence of Petviashvili's Method near Periodic Waves in the Fractional Korteweg--de Vries Equation.} SIAM J. Math. Anal., 51(4), pp.2850–2883.

\bibitem{petviashvili}
Petviashvili, V.I., (1976), \emph{Equation of an extraordinary soliton}. Fizika plazmy, 2, pp.469–472.




\bibitem{struwe}
Struwe, M., (2008), \emph{Variational methods.}, Vol 34 of Ergebnisse der Mathematik und ihrer Grenzgebiete. 3. Folge. A series of Modern Surveys in Mathematics, 4th edn, Springer, Berlin. 









\end{thebibliography}
\end{document}